%% file: Lu-Tan.tex
\documentclass{amsart}
\usepackage{cases}
\usepackage{graphicx}
\usepackage{amssymb}
\usepackage{amsfonts}
\usepackage{amsmath}
\usepackage{amscd}
\usepackage{xypic}
\usepackage{xcolor}
 \newtheorem{thm}{Theorem}[section]
 \newtheorem{cor}[thm]{Corollary}
 \newtheorem{lem}[thm]{Lemma}
 \newtheorem{prop}[thm]{Proposition}
 \theoremstyle{definition}
 \newtheorem{defn}[thm]{Definition}
 
 \theoremstyle{remark}
 \newtheorem{rem}{Remark}
 \newtheorem{example}{Example}[section]
 \numberwithin{equation}{section}
 
 \DeclareMathOperator{\IM}{Im}

 \newcommand{\Hom}{\mathrm{Hom}}

 \newcommand{\Z}{\mathbb{Z}}

 \def\H{\text{Hom}}

\begin{document}
\title[Small covers and the bordism classification of 2-torus manifolds]
 {Small covers and the equivariant bordism classification of 2-torus manifolds}

\author{Zhi L\"u and Qiangbo Tan}

\address{School of Mathematical Sciences, Fudan University \\ Shanghai,
200433, People's Republic of China.}
 \email{zlu@fudan.edu.cn}
 \address{School of Mathematical Sciences \\ Fudan University \\ Shanghai,
200433, People's Republic of China.}
 \email{081018011@fudan.edu.cn}
\thanks{This work is supported by grants from NSFC (No. 10931005),  Shanghai NSF (No. 10ZR1403600) and RFDP (No. 20100071110001).}


\begin{abstract}
Associated with the Davis--Januszkiewicz theory of small covers,
this paper deals with the theory of 2-torus manifolds from the
viewpoint of equivariant bordism. We define a differential operator
on the ``dual'' algebra of the unoriented $G_n$-representation
algebra introduced by Conner and Floyd,  where $G_n=(\Z_2)^n$. With
the help of $G_n$-colored graphs (or mod 2 GKM graphs), we may use
this differential operator to give a very simple description of tom
Dieck--Kosniowski--Stong localization theorem in the setting of
2-torus manifolds. We then apply this to study  the
$G_n$-equivariant unoriented bordism classification of
$n$-dimensional 2-torus manifolds. We show that the
$G_n$-equivariant unoriented bordism class of each $n$-dimensional
2-torus manifold contains an $n$-dimensional small cover as its
representative, solving the conjecture posed in \cite{l1}. In
addition, we also obtain that the graded noncommutative ring  formed
by the equivariant unoriented bordism classes of  2-torus manifolds
of all possible dimensions is generated by the classes of all
generalized real Bott manifolds (as special small covers over the
products of simplices). This gives a strong connection between the
computation of $G_n$-equivariant bordism groups or ring and the
Davis--Januszkiewicz theory of small covers. As a computational
application, with the help of computer,  we completely determine the
structure of the group formed by equivariant bordism classes of all
4-dimensional 2-torus manifolds. Finally, we give some essential
relationships among 2-torus manifolds, coloring polynomials, colored
simple convex polytopes, colored graphs.
 \end{abstract}

\maketitle

\section{Introduction}\label{introduction}
Throughout this paper, assume that $G_n=({\Bbb Z}_2)^n$ is the mod
2-torus group of rank $n>0$. An $n$-dimensional {\em 2-torus
manifold} is a smooth closed $n$-dimensional (not necessarily
oriented) manifold equipped with an  effective smooth  $G_n$-action,
so its  fixed point set  is empty or consists
 of a set of isolated points (see \cite{l1, lm}).
 The seminal work of Davis and Januszkiewicz in
\cite{dj} discussed a special kind of  2-torus manifolds, so called
{\em small covers} (as the topological versions of real toric
varieties), each of which admits a locally standard $G_n$-action
such that the orbit space of action is homeomorphic to a simple
convex polytope. Small covers provide a strong link between
equivariant topology, polytope theory and combinatorics. This paper
will deal with the theory of 2-torus manifolds from the viewpoint of
equivariant bordism. We still expect such a link like small covers
although 2-torus manifolds form a much wider class than small
covers.

\vskip .1cm

In 1960s, Conner and Floyd began with the study of the equivariant
bordism theory of smooth closed $G_n$-manifolds (\cite{cf, co}).
They introduced and studied a graded commutative algebra  over
${\Bbb Z}_2$  with unit, $\mathcal{Z}_*(G_n)=\sum_{m\geq
0}\mathcal{Z}_m(G_n)$, where $\mathcal{Z}_m(G_n)$ consists of
$G_n$-equivariant unoriented bordism  classes of all smooth closed
$m$-dimensional manifolds $M^m$ with smooth  $G_n$-actions fixing a
finite
 set, and  the addition and the multiplication on $\mathcal{Z}_*(G_n)$ are
 defined by the disjoint union and the cartesian
 products with diagonal $G_n$-actions of $G_n$-manifolds, respectively.
 We note that $\mathcal{Z}_m(G_n)$ has also  a
$\Z_2$-linear space structure. 
 As stated in \cite[p.75]{cf} and \cite[p.107]{co},
 an action of  $G_n$ on $M^m $  is equivalent to a collection of involutions
 $T_i:M^m\longrightarrow M^m, i=1,..., n$, with $T_iT_j=T_jT_i$, which means that $G_n$ is
 generated by $T_1, ..., T_n$. We also require the $G_n$-action to be effective.
 Thus, $\mathcal{Z}_n(G_n)$ consists of  equivariant unoriented bordism classes of
all $n$-dimensional 2-torus manifolds, also denoted by
$\mathfrak{M}_n$ in \cite{l1}.

\begin{rem} Conner and Floyd showed in \cite{cf, co} that
when $n=1$ $\mathcal{Z}_*(\Z_2)\cong\Z_2$, and when $n=2$,
$\mathcal{Z}_*((\Z_2)^2)\cong\Z_2[u]$ where $u$ denotes the class of
${\Bbb R}P^2$ with the standard  $(\Z_2)^2$-action. When $n=3$, the
group structure of $\mathcal{Z}_3((\Z_2)^3)=\mathfrak{M}_3$ was
determined in \cite{l1} (see also \cite{ly}), and it was also shown
therein that $\dim_{\Z_2}\mathcal{Z}_3((\Z_2)^3)=13$. However, as
far as authors know, when $n\geq 3$, the ring structure of
$\mathcal{Z}_*(G_n)$ is still open, and the group structure of
$\mathcal{Z}_n(G_n)$ with $n>3$ is also so.
\end{rem}

With the above understood, now our purpose  in this paper can be
equivalently restated as follows: we shall concentrate on the study
of $\mathcal{Z}_n(G_n)$ as a group generated by equivariant
unoriented bordism classes of all $n$-dimensional 2-torus manifolds.
Given a nonzero class $\{M^n\}$ in $\mathcal{Z}_n(G_n)$, we know  by
classical results of Stong \cite{s} that the fixed point set of
$M^n$ is a nonempty finite set,  and the collection of the
associated tangential $G_n$-representations at the fixed points,
which is considered as a squarefree homogenous polynomial of degree
$n$ in $\mathcal{R}_*(G_n)$, determines the class $\{M^n\}$,  where
$\mathcal{R}_*(G_n)=\sum_{m\geq 0}\mathcal{R}_m(G_n)$ is the graded
polynomial algebra
 over $\Z_2$ generated by all 
irreducible $G_n$-representations. In other words, there is a
monomorphism $\phi_n:\mathcal{Z}_n(G_n)\longrightarrow
\mathcal{R}_n(G_n)$ (see also Subsection~\ref{localization}). We
shall consider the following questions:
\begin{enumerate}
\item[(Q1)] What squarefree homogenous polynomials in $\mathcal{R}_n(G_n)$ arise
as fixed point data of 2-torus manifolds?
\item[(Q2)] Are there preferred representatives in the equivariant unoriented bordism
classes of $\mathcal{Z}_n(G_n)$? With respect to this question, the
following concrete conjecture was posed in \cite{l1}.
\end{enumerate}
\noindent {\bf Conjecture ($\star$):}  Each class of
$\mathcal{Z}_n(G_n)$ contains a small cover as its representative.

\vskip .1cm

 On (Q1), there is a theoretical answer stemming from
the tom Dieck--Kosniowski--Stong localization theorem in terms of an
integrality condition for the tangential fixed point data, which
applies to smooth closed $G_n$-manifolds of an arbitrary dimension
$m$ (not necessarily equal to $n$) with finite fixed point set. In
the special case of $m=n$ (i.e., 2-torus manifolds), we shall
formulate a simple criterion in terms of the vanishing of a
differential on the dual of the given squarefree homogenous
polynomial (see Theorem~\ref{main result}). This description is
based upon the consideration of $G_n$-colored graphs or mod 2 GKM
graphs, which was introduced and studied in \cite{bl, bgh, l2, l3}.
Furthermore, we may use this simple criterion to consider (Q2). We
show  that each $n$-dimensional 2-torus manifold is
$G_n$-equivariantly bordant to an $n$-dimensional small cover,
giving an affirmative answer to Conjecture ($\star$) (see
Theorem~\ref{conj}). As shown in \cite{cms}, a small cover over a
product of simplices is a generalized real Bott manifold. We also
show that as a graded noncommutative  ring,
 $\mathfrak{M}_*=\sum_{n\geq 1}\mathcal{Z}_n(G_n)$ is generated by the classes of all
generalized real Bott manifolds (see Theorem~\ref{ring}), where the
multiplication on $\mathfrak{M}_*$ is defined by the cartesian
product of actions (see Subsection~\ref{state} for details). This
provides a strong link between the computation of $G_n$-equivariant
bordism groups or ring and the Davis--Januszkiewicz theory of small
covers.
 As a
computational application, we use a computer program to get that if
$n=4$, then the $\Z_2$-linear space $\mathcal{Z}_4(G_4)$ produced by
equivariant bordism classes of all 4-dimensional 2-torus manifolds
has dimension 510 (see Proposition~\ref{dim-4}).

\vskip .1cm This paper is organized as follows. We shall review the
tom Dieck--Kosniowski--Stong localization theorem and give the
statements of main results in Section~\ref{sec2}.  In
Section~\ref{dual-diff} we introduce the notions of faithful
polynomials and its dual polynomials, and give the definition of the
differential operator $d$ on the free polynomial algebra
$\Z_2[\widehat{\Hom(\Z_2, G_n)}]$. In Section~\ref{colored
graph-small cover} we review the basic theories of colored graphs
and small covers. In particular,  we also discuss the
decomposability of $G_n$-colored simple convex polytopes. In
Section~\ref{proof of main} we give the proof of Theorem~\ref{main
result}. In Section~\ref{app-1} we introduce the linear spaces
$\mathcal{V}_n$ and $\mathcal{V}_n^*$, and then use them to study
the structure of $\mathfrak{M}_*$ and to finish the proofs of
Theorems~\ref{conj}--\ref{ring}. In Section~\ref{app-2} we give a
summary on some essential relationships among 2-torus manifolds,
coloring polynomials, colored simple convex polytopes, colored
graphs, and also pose some problems. In Section~\ref{re-proof},  for a local completeness we give a simple  proof to show how the  tom Dieck--Kosniowski--Stong localization theorem follows from the existence theorem of tom Dieck.
Finally we give an  algorithm
of determining  a basis  of $\mathcal{V}_4^*$ in
Section~\ref{computer}.

\section{tom Dieck--Kosniowski--Stong localization theorem and
statements of main results}\label{sec2}

\subsection{tom Dieck--Kosniowski--Stong localization theorem}
\label{localization}

Following~\cite{cf, co}, let $\mathcal{R}_m(G_n)$ be the linear
space over ${\Bbb Z}_2$,
 generated by the isomorphism classes of $m$-dimensional real $G_n$-representations.
 Then $\mathcal{R}_*(G_n)=\sum_{m\geq 0}\mathcal{R}_m(G_n)$ becomes a graded
 commutative algebra over ${\Bbb Z}_2$ with unit,
  called the {\em Conner--Floyd unoriented $G_n$-representation algebra} here,
   where the multiplication is given by
  the direct sum of representations, and the unit is given by the representation of degree 0.
  As pointed out in \cite{co},  $\mathcal{R}_*(G_n)$ is not the Grothendieck ring of
  $G_n$-representations, an entirely different concept.
  It is well-known that each irreducible real $G_n$-representation is one-dimensional,
  so $\mathcal{R}_*(G_n)$ is also the
  graded polynomial algebra over $\Z_2$ generated by the  isomorphism classes of
  one-dimensional irreducible
  real $G_n$-representations. There is the following essential relation between
  $\mathcal{Z}_*(G_n)$ and $\mathcal{R}_*(G_n)$, due to Stong \cite{s} (see also \cite{co} and \cite{ms}).
  \begin{thm}[Stong]\label{stong} The map
$$\phi_*: \mathcal{Z}_*(G_n)\longrightarrow \mathcal{R}_*(G_n)$$ defined
 by $\{M\}\longmapsto \sum_{p\in M^{G_n}}[\tau_pM]$ is
a monomorphism as algebras over ${\Bbb Z}_2$ where $\tau_pM$ denotes
the real $G_n$-representation on the tangent space at $p\in
M^{G_n}$.
\end{thm}
Thus, $\mathcal{Z}_*(G_n)$ is identified with a subalgebra $\IM
\phi_*$ of $\mathcal{R}_*(G_n)$, also denoted by
$\mathcal{S}_*(G_n)$ in~\cite{cf, co}.

\begin{rem}
Given a class $\{M^m\}$ in $\mathcal{Z}_m(G_n)$, as mentioned in
Section~\ref{introduction}, the action of $G_n$ on $M^m$ is
effective. If the fixed point set $M^{G_n}$ of the $G_n$-action on
$M^m$ is empty, then $\{M^m\}=0$ in $\mathcal{Z}_m(G_n)$ by a result
of Stong~\cite{s} or \cite[Theorem 31.2]{co}. Thus, if $\{M^m\}$ is
nonzero in $\mathcal{Z}_m(G_n)$, then  $m\geq n$ since the action of
$G_n$ on $M^m$ is effective. This implies that if $m<n$, then
$\mathcal{Z}_m(G_n)$ would be trivial. It should be pointed out that
in~\cite{co}, the class of a single point with a $G_n$-action is
used as  the unit  in $\mathcal{Z}_*(G_n)$, which represents the
generator of $\mathcal{Z}_0(G_n)$ so $\mathcal{Z}_0(G_n)\cong \Z_2$.
This enriches the algebraic structure of $\mathcal{Z}_*(G_n)$, and
does not bring any essential influence on the study of
$\mathcal{Z}_*(G_n)$ although a $G_n$-action on a single point is
trivial. In this paper we use this convention in \cite{co}.
\end{rem}

In \cite{d}, tom Dieck showed that the equivariant unoriented
bordism class of a smooth closed $G_n$-manifold is completely
determined by its equivariant Stiefel--Whitney characteristic
numbers, and in particular, the existence of a $G_n$-manifold $M^m$
fixing a finite set can be characterized by the integral property of
its fixed point data (see \cite[Theorem 6]{d}). Later on, Kosniowski and
Stong \cite{ks} gave a more precise localization formula for the
characteristic numbers of $M^m$ in terms of the fixed point data. Then the
existence theorem of tom Dieck can be  formulated  into  the
following localization theorem in terms of Kosniowski and Stong's
localization formula (see \cite[\S5 of p.740]{ks}). Here for a  completeness, we shall give a simple proof for how the following theorem follows from the existence theorem of tom Dieck in Section~\ref{re-proof} of this paper.
\begin{thm}  [tom Dieck--Kosniowski--Stong localization theorem]\label{dks}
Let $\{\tau_1, ..., \tau_l\}$ be a collection of faithful
$G_n$-representations in $\mathcal{R}_m(G_n)$. Then a necessary and
sufficient condition that $\tau_1+\cdots+\tau_l\in \IM\phi_m$ $($or
$\{\tau_1, ..., \tau_l\}$ is the fixed point data of a
$G_n$-manifold $M^m)$ is that for all symmetric polynomial functions
$f(x_1,...,x_m)$ over ${\Bbb Z}_2$,
\begin{equation} \label{formula-tks1}\sum_{i=1}^l{{f(\tau_i)}\over{\chi^{G_n}(\tau_i)}}\in
H^*(BG_n;{\Bbb Z}_2)\end{equation}  where $\chi^{G_n}(\tau_i)$
denotes the equivariant Euler class of $\tau_i$, which is a product
of $m$ nonzero elements of $H^1(BG_n;{\Bbb Z}_2)$, and $f(\tau_i)$
means that variables $x_1,...,x_m$ in the function $f(x_1,...,x_m)$
are replaced by those $m$ degree-one factors in
$\chi^{G_n}(\tau_i)$.
\end{thm}

\begin{rem}
  Although all elements of $\IM \phi_*$ can be characterized by the formula (\ref{formula-tks1}),
  it is still
quite difficult to determine the algebra structure of $\IM
\phi_*\cong \mathcal{Z}_*(G_n)$. Also, in Theorem~\ref{dks}, if
$\{\tau_1, ..., \tau_l\}$ is the fixed data of a  $G_n$-manifold
$M^m$, then the polynomial (\ref{formula-tks1}) is exactly an
equivariant Stiefel--Whitney number of $M^m$. Actually, if we
formally write the equivariant total Stiefel--Whitney class of the
tangent bundle $\tau M^m$ as $w^{G_n}(\tau
M^m)=\prod_{i=1}^m(1+x_i)$, then the equivariant Stiefel--Whitney
number $f(x_1, ..., x_m)[M^m]$ can be calculated by the  formula
$$f(x_1, ..., x_m)[M^m]=\sum_{i=1}^l{{f(\tau_i)}\over{\chi^{G_n}(\tau_i)}}\in H^*(BG_n;\Z_2)$$
where $[M^m]$
denotes the fundamental homology class of $M^m$.  For more details,
see \cite{d} and \cite{ks}.
\end{rem}

\subsection{Statements of main results}\label{state}
Now let $\Hom(G_n,\Z_2)$ (resp. $\Hom(\Z_2, G_n)$) denote the set of
all homomorphisms $G_n\longrightarrow\Z_2$ (resp.
$\Z_2\longrightarrow G_n$). Then both $\Hom(G_n,\Z_2)$ and
$\Hom(\Z_2, G_n)$ have natural abelian group structures given by
those of $\Z_2$ and $G_n$ in the usual  way (i.e., the addition is
given by $(\rho+\xi)(g)=\rho(g)+\xi(g)$) and they have also linear
space structures over $\Z_2$.
Let $\Z_2[\Hom(G_n, \Z_2)]$ (resp. $\Z_2[\Hom(\Z_2, G_n)]$)  be the
graded polynomial algebra over the linear space $\Hom(G_n,\Z_2)$
(resp. $\Hom(\Z_2, G_n)$), i.e., the infinite symmetric tensor
algebra over $\Hom(G_n,\Z_2)$ (resp. $\Hom(\Z_2, G_n)$).
 Since both $\Hom(G_n,\Z_2)$ and  $\Hom(\Z_2, G_n)$ are isomorphic to $H^1(BG_n;\Z_2)$
 as linear spaces,
we have that  $$\Z_2[\Hom(G_n, \Z_2)] \cong \Z_2[\Hom(\Z_2, G_n)]
\cong H^*(BG_n;\Z_2)$$ as algebras.

\vskip .1cm

 On the other hand, it is well-known that all irreducible
real $G_n$-representations bijectively correspond to all  elements
in $\H(G_n,{\Bbb Z}_2)$, where every irreducible real representation
of $G_n$ has the form $\lambda_\rho: G_n\times{\Bbb
R}\longrightarrow{\Bbb R}$ with $\lambda_\rho(g,x)=(-1)^{\rho(g)}x$
for $\rho\in\H(G_n,{\Bbb Z}_2)$, and $\lambda_\rho$ is trivial if
$\rho(g)=0$ for all $g\in G_n$. Thus, if by $\Z_2[\widehat{\Hom(G_n,
\Z_2)}]$ we denote the free polynomial algebra
 over $\widehat{\Hom(G_n, \Z_2)}$,  then $\Z_2[\widehat{\Hom(G_n, \Z_2)}]$ can be identified with
 $\mathcal{R}_*(G_n)$, where $\widehat{\Hom(G_n, \Z_2)}$  means the set obtained
 by forgetting the algebraic structure on $\Hom(G_n, \Z_2)$.
 Similarly, we may define $\Z_2[\widehat{\Hom(\Z_2, G_n)}]$ in the same way as
 $\Z_2[\widehat{\Hom(G_n, \Z_2)}]$.
 We note that $\Z_2[\widehat{\Hom(G_n, \Z_2)}]$ is generated by $2^n$  elements of
  $\widehat{\Hom(G_n, \Z_2)}$, while $\Z_2[\Hom(G_n, \Z_2)]$ is generated by a basis
  (containing $n$ elements) of
   $\Hom(G_n, \Z_2)$.
\vskip .1cm

 In a certain sense, both $\Hom(G_n,\Z_2)$ and $\Hom(\Z_2, G_n)$ are dual to each other.
Thus, given a faithful $G_n$-polynomial $g=\sum_i t_{i,1}\cdots
t_{i, n}$ in $\Z_2[\widehat{\Hom(G_n, \Z_2)}]$ (which means that for
each monomial $t_{i,1}\cdots t_{i, n}$ of $g$, the set $\{t_{i,1},
..., t_{i, n}\}$ is a basis of $\Hom(G_n,\Z_2)$), we can obtain a
unique dual $G_n$-polynomial $g^*$ in $\Z_2[\widehat{\Hom(\Z_2,
G_n)}]$ (see also Subsection~\ref{dual poly}).  In
Subsection~\ref{operator} we shall define a differential operator
$d$ on $\Z_2[\widehat{\Hom(\Z_2, G_n)}]$. Identifying
$\Z_2[\widehat{\Hom(G_n, \Z_2)}]$  with
 $\mathcal{R}_*(G_n)$, we may regard $\IM \phi_*$  as a subring of $\Z_2[\widehat{\Hom(G_n, \Z_2)}]$.
Then the following result gives
another characterization of $g\in \IM \phi_n$ in terms of $d(g^*)$.

\begin{thm}\label{main result}
Let $g=\sum_i t_{i,1}\cdots t_{i, n}$ be a faithful $G_n$-polynomial
in $\Z_2[\widehat{\Hom(G_n, \Z_2)}]$. Then
 $g\in \IM \phi_n$ if and only if $d(g^*)=0$.
\end{thm}

\begin{rem}
We shall see from Theorem~\ref{color poly} in
Subsection~\ref{graph1} that $g\in \IM \phi_n$ can also be
characterized by a $G_n$-colored graphs (or mod 2 GKM graph), so
that we may use the $G_n$-colored graphs  to give the proof of
Theorem~\ref{main result}. On the other hand, there is also an
essential relation between $G_n$-colored graphs and $G_n$-colored
simple convex $n$-polytopes in the setting of small covers, which
indicates an
algebraic duality (see Proposition~\ref{small}). This is an important reason why  
we consider the dual polynomial $g^*$ of $g\in \IM \phi_n$.
\end{rem}

Since each class $\chi^{G_n}(\tau_i)$ uniquely corresponds to
$\tau_i$ in Theorem~\ref{dks} and $H^*(BG_n;\Z_2)$ is isomorphic to
$\Z_2[\Hom(G_n, \Z_2)]$, as a consequence of
Theorems~\ref{dks}--\ref{main result}, we have the following
interesting algebraic corollary.

\begin{cor}\label{diff-formula}
Let $g=\sum_i t_{i,1}\cdots t_{i, n}$ be a faithful $G_n$-polynomial
in $\Z_2[\widehat{\Hom(G_n, \Z_2)}]$.
 Then $d(g^*)=0$ if and only if for all symmetric
polynomial functions $f(x_1,...,x_n)$ over ${\Bbb Z}_2$,
\begin{equation*}\label{integral}
\sum_{i}{{f(t_{i,1}, ..., t_{i, n})}\over{t_{i,1}\cdots t_{i,
n}}}\in\Z_2[\Hom(G_n, \Z_2)]\end{equation*} when $t_{i,1}\cdots
t_{i, n}$ and $f(t_{i,1}, ..., t_{i, n})$ are regarded as
polynomials in $\Z_2[\Hom(G_n, \Z_2)]$.
\end{cor}

Our next task is to apply Theorem~\ref{main result} to the study of
$\mathfrak{M}_*$.

\vskip .2cm

  Since $\phi_n: \mathcal{Z}_n(G_n)\longrightarrow \mathcal{R}_n(G_n)$ is  a monomorphism,
  it follows by
  Theorem~\ref{main result} that as linear spaces over $\Z_2$, $\mathcal{Z}_n(G_n)$ is
  isomorphic to the linear space
  $\mathcal{V}_n$ formed by all faithful $G_n$-polynomials $g\in\Z_2[\widehat{\Hom(G_n, \Z_2)}]$
  with $d(g^*)=0$.
  Then, the problem can be further reduced to studying the linear space
   $\mathcal{V}^*_n$ formed by the dual polynomials of those polynomials in
$\mathcal{V}_n$ (see Section~\ref{app-1}). Based upon this and the
Davis--Januszkiewicz theory of small covers, we will show that

\begin{thm} \label{conj}
The Conjecture $(\star)$ holds for arbitrary  dimension $n$.
\end{thm}

\begin{rem} It has been shown in \cite{l1} that the Conjecture ($\star$) holds for $n\leq 3$.
     However, the argument used in \cite{l1} does not work effectively  in the case $n>3$.
          \end{rem}

The sum $\mathfrak{M}_*=\sum_{n\geq 1}\mathcal{Z}_n(G_n)$ is also a
graded ring with
 the multiplication defined by $\{M_1^{n_1}\}\cdot\{M_2^{n_2}\}=\{M_1^{n_1}\times
M_2^{n_2}\}$ where  the $(\Z_2)^{n_1+n_2}$-action on
$M_1^{n_1}\times M_2^{n_2}$ is given by $\big((g_1, g_2), (x_1,
x_2)\big)\longmapsto \big(g_1(x_1), g_2(x_2)\big)$ 
by regarding $(\Z_2)^{n_1+n_2}$
as $(\Z_2)^{n_1}\times (\Z_2)^{n_2}$. It should be pointed out that
the  multiplication defined as above depends upon the ordering of
the cartesian  product of $M_1^{n_1}$ with $(\Z_2)^{n_1}$-action and
$M_2^{n_2}$ with
  $(\Z_2)^{n_2}$-action.  Actually,  in the same way as above, by regarding $(\Z_2)^{n_1+n_2}$ as $(\Z_2)^{n_2}\times (\Z_2)^{n_1}$, the $(\Z_2)^{n_1+n_2}$-action on $M_2^{n_2}\times M_1^{n_1}$ would be defined  by $\big((g_2, g_1), (x_2, x_1)\big)\longmapsto \big(g_2(x_2), g_1(x_1)\big)$. However, generally such two $(\Z_2)^{n_1+n_2}$-actions on $M_1^{n_1}\times M_2^{n_2}$
  and $M_2^{n_2}\times M_1^{n_1}$  are not equivariantly cobordant except for $\{M_1^{n_1}\}=\{M_2^{n_2}\}$,
  but up to automorphisms of $(\Z_2)^{n_1+n_2}$, they have not any difference essentially
  (i.e., by using an automorphism of $(\Z_2)^{n_1+n_2}$, one of such two actions can be changed into the other one).
  Thus, $\mathfrak{M}_*$ is  a graded noncommutative  ring.

\begin{thm}\label{ring}
$\mathfrak{M}_*$ is generated by the classes of all small covers over $\Delta^{n_1}\times\cdots\times\Delta^{n_\ell}$
with $n_1+\cdots+n_\ell\geq 1$, where $\Delta^{n_i}$ is an $n_i$-simplex.
\end{thm}

\begin{rem}
We know from \cite{cms} that a small cover over a product of simplices is actually identified with a generalized
real Bott manifold, so $\mathfrak{M}_*$ is generated by the classes of all generalized
real Bott manifolds.
\end{rem}

As a computational application, we determine the precise structure
of $\mathcal{Z}_4(G_4)$.

\begin{prop}\label{dim-4}
$\mathcal{Z}_4(G_4)$ is generated by merely the classes of small
covers over $\Delta^2\times \Delta^2$, and has dimension $510$.
\end{prop}
In addition, we shall also give a simple proof of the main result on
$\mathcal{Z}_3(G_3)$ in \cite{l1} (see Remark~\ref{m3}).

\section{Faithful polynomials, dual polynomials and a differential  operator} \label{dual-diff}

\subsection{Faithful polynomials and dual polynomials}\label{dual poly}
 $\Hom (G_n,\Z_2)$ and $\Hom (\Z_2, G_n)$ are clearly isomorphic to $G_n$,
 and they are dual to each other by the following pairing:
\begin{equation}\label{pairing}
\langle\cdot, \cdot\rangle: \Hom(\Z_2, G_n)\times\Hom(G_n,
\Z_2)\longrightarrow \Hom(\Z_2,\Z_2)
\end{equation}
defined by $\langle \xi, \rho\rangle=\rho\circ \xi$, composition of
homomorphisms. For example, the standard basis $\{\rho_1, ..., $
$\rho_n\}$ of $\Hom(G_n,\Z_2)$ gives the dual basis
$\{\rho_1^*,...,\rho_n^*\}$ of $\Hom(\Z_2,G_n)$, where $\rho_i$ is
defined by $(g_1,..., g_n)\longmapsto  g_i$, and $\rho_i^*$ is
defined by $a\longmapsto (\underbrace{0,...,0}_{i-1}, a, 0,...,0)$.

Recall that $\Z_2[\widehat{\Hom(G_n, \Z_2)}]$ can be identified with
 $\mathcal{R}_*(G_n)$, such that each monomial in $\Z_2[\widehat{\Hom(G_n,
 \Z_2)}]$ can be used as the class of a $G_n$-representation in
 $\mathcal{R}_*(G_n)$.
  Suppose that $g =\sum_it_{i,1}\cdots t_{i,n}$ is a nonzero homogeneous polynomial of degree $n$ in
  $\Z_2[\widehat{\H(G_n,\Z_2)}]$ such that each monomial $t_{i,1}\cdots t_{i,n}$ is the class of
  an $n$-dimensional
  faithful $G_n$-representation,  so  $\{t_{i,1}, ...,  t_{i,n}\}$ forms a basis of $\H(G_n,\Z_2)$.
  Such a polynomial $g$ is  called a
 {\em faithful $G_n$-polynomial} of $\Z_2[\widehat{\H(G_n,\Z_2)}]$.  By the pairing (\ref{pairing}),
 $\{t_{i,1}, ...,  t_{i,n}\}$ determines a dual basis $\{s_{i,1},..., s_{i,n}\}$ of $\H(\Z_2, G_n)$. Furthermore, we obtain a unique  homogeneous polynomial $g^*=\sum_i s_{i,1}\cdots s_{i,n}$ in
$\Z_2[\widehat{\H(\Z_2,G_n)}]$, which is called the {\em dual
$G_n$-polynomial} of $g$.

\begin{example}\label{example-dual}
When $n=3$, take a faithful polynomial
$g=\rho_1\rho_2\rho_3+\rho_1\rho_3(\rho_2+\rho_3)+
\rho_1\rho_2(\rho_2+\rho_3)+\rho_1(\rho_1+\rho_3)(\rho_1+\rho_2)+\rho_1(\rho_1+\rho_3)(\rho_2+\rho_3)
+\rho_1(\rho_1+\rho_2)(\rho_2+\rho_3)$ in
$\Z_2[\widehat{\H(G_3,\Z_2)}]$. Then the dual polynomial of $g$ is
$g^*=\rho_1^*\rho_2^*\rho_3^*+\rho_1^*\rho_2^*(\rho_2^*+\rho_3^*)+\rho_1^*\rho_3^*(\rho_2^*+\rho_3^*)+\rho_2^*\rho_3^*(\rho_1^*+\rho_2^*+\rho_3^*)+
\rho_2^*(\rho_2^*+\rho_3^*)(\rho_1^*+\rho_2^*+\rho_3^*)+\rho_3^*(\rho_2^*+\rho_3^*)(\rho_1^*+\rho_2^*+\rho_3^*)$
in $\Z_2[\widehat{\H(\Z_2, G_3)}]$.
\end{example}

\subsection{A differential operator $d$ on $\Z_2[\widehat{\H(\Z_2, G_n)}]$}\label{operator}
We define a differential operator $d$ on $\Z_2[\widehat{\H(\Z_2,
G_n)}]$ as follows: for each monomial $s_1\cdots s_i$ of degree
$i\geq 1$
$$d_i(s_1\cdots s_i)=\begin{cases}
\sum_{j=1}^is_1\cdots s_{j-1}\widehat{s}_js_{j+1}\cdots s_i &\text{ if } i>1\\
1 &\text{ if } i=1.
\end{cases}$$
and $d_0(1)=0$, where the symbol $\widehat{s}_j$ means that $s_j$ is
deleted. Obviously, $d^2=0$. Thus, $(\Z_2[\widehat{\H(\Z_2, G_n)}],
d)$ forms a chain complex.

\begin{prop}\label{acyclic}
For all $i\geq 0$,
 $H_i(\Z_2[\widehat{\text{\rm \H}(\Z_2, G_n)}];{\Bbb Z}_2)=0$.
\end{prop}
\begin{proof} It is easy to see that $H_0(\Z_2[\widehat{\H(\Z_2, G_n)}];{\Bbb Z}_2)=0$.
So it suffices to show that $\text{\rm Im} d_{i+1}=\ker d_i$ for
$i>0$. Obviously, $\text{\rm Im} d_{i+1}\subseteq\ker d_i$.
Conversely, for any $h\in \ker d_i$, take $\varphi=sh$ where
$s\in\widehat{\H(\Z_2, G_n)}$. Then $d_{i+1}(\varphi)=h+sd_i(h)=h$
so $h\in \text{\rm Im} d_{i+1}$. Thus $\text{\rm Im}
d_{i+1}\supseteq\ker d_i$.
\end{proof}

\begin{defn}
A polynomial $h\in \Z_2[\widehat{\H(\Z_2, G_n)}]$ is said to be {\em
squarefree} if each monomial of $h$ is a product of distinct
nontrivial elements in $\widehat{\H(\Z_2, G_n)}$, where the trivial
element in $\widehat{\H(\Z_2, G_n)}$ is the zero homomorphism from
$\Z_2$ to $G_n$.
\end{defn}

\begin{cor}\label{diff}
Let $h\in \Z_2[\widehat{\Hom(\Z_2, G_n)}]$ be squarefree. Then
$d(h)=0$ if and only if there is a squarefree polynomial $\varphi$
in $\Z_2[\widehat{\Hom(\Z_2, G_n)}]$ such that $d(\varphi)=h$.
\end{cor}

\begin{proof}
Obviously, if $h=d(\varphi)$, then $d(h)=d^2(\varphi)=0$.
Conversely,  by Leibniz rule, for any nontrivial element $s\in
\widehat{\Hom(\Z_2, G_n)}$,  $d(sh)=h+sd(h)=h$. If $sh$ is not
squarefree, then we may write $sh=sh_1+s^2h_2$ such that $sh_1$ is
nonzero and sequarefree. Furthermore,
$h=d(sh)=h_1+sd(h_1)+s^2d(h_2)$. This forces $d(h_2)$ to be zero
since $h$ is squarefree. Thus, we can take $\varphi=sh_1$ as
desired.
\end{proof}

\begin{rem}
It should be pointed out that similarly we may define a differential
operator $d'$ on $\Z_2[\widehat{\H(G_n,\Z_2)}]$. However, given a
faithful polynomial $g\in \Z_2[\widehat{\H(G_n,\Z_2)}]$, if
$d(g^*)=0$, then generally we cannot obtain $d'(g)=0$. For example,
for the $g$ and $g^*$  in Example~\ref{example-dual},  a direct
calculation shows that $d(g^*)=0$ but $d'(g)\not=0$.
\end{rem}

Let $h\in \Z_2[\widehat{\H(\Z_2, G_n)}]$. For an automorphism
$\sigma$ of $\Hom(\Z_2, G_n)$, let $\sigma(h)$ denote the polynomial
of $\Z_2[\widehat{\H(\Z_2, G_n)}]$ produced by replacing each
degree-one factor $t$ in $h$ by $\sigma(t)$, where $t$ is regarded
as an element in $\Hom(\Z_2, G_n)$.  Then we see that
$\Z_2[\widehat{\H(\Z_2, G_n)}]$ naturally admits an action $\Phi$ of
$\text{Aut}(\Hom(\Z_2, G_n))$, defined by
 $h\longmapsto\sigma(h)$. A direct calculation gives the following result.
 \begin{lem}\label{commut}
 Let $h\in \Z_2[\widehat{\Hom(\Z_2, G_n)}]$ and $\sigma\in\text{\rm Aut}(\Hom(\Z_2, G_n))$.
 Then $$d(\sigma(h))=\sigma(d(h)).$$
 \end{lem}

\section{$G_n$-colored graphs and small covers}\label{colored graph-small cover}

Throughout the following, $\Z_2[\widehat{\Hom(G_n, \Z_2)}]$ will be
identified with $\mathcal{R}_*(G_n)$. Then we may write the Stong
homomorphism as $\phi_*: \mathcal{Z}_*(G_n)\longrightarrow
\Z_2[\widehat{\Hom(G_n, \Z_2)}]$.

\subsection{$G$-colored graphs} \label{graph1}
In \cite{gkm}, Goresky, Kottwitz and MacPherson established the GKM
theory, indicating that  there is an essential link between topology
and geometry of torus actions and the combinatorics of colored
graphs (see also \cite{gz}). Such a link has already been expanded
to the case of mod 2-torus actions (see, e.g.,
\cite{bl}--\cite{bgh}, \cite{l2}, and \cite{l3}). Specifically,
assume that $M^m$ is a smooth closed $m$-manifold with an effective
smooth $G_n$-action fixing a nonempty finite set $M^{G_n}$, which
implies $m\geq n$ (see \cite{ap}). Then we know from \cite{l2, l3}
that the $G_n$-action on $M^m$ defines a regular graph $\Gamma_M$ of
valence $m$ with the vertex set $M^{G_n}$ and a $G_n$-coloring
$\alpha$.

\vskip .1cm In this paper we shall pay more attention on the extreme
case in which $m=n$ (i.e., $M^n$ is a 2-torus manifold). In this
case we also know from \cite{bl, l2, l3} that such a $G_n$-colored
graph $(\Gamma_M, \alpha)$ is uniquely determined by the
$G_n$-action where $\alpha$ is defined as
 a  map  from the
set $E_{\Gamma_M}$ of all edges of $\Gamma_M$ to all non-trivial
elements of $\Hom(G_n,{\Bbb Z}_2)$, and it satisfies the following
properties:.
\begin{enumerate}
\item[(P1)] for each vertex $v$ of $\Gamma_M$, $\prod_{x\in E_v}\alpha(x)$ is faithful
in $\Z_2[\widehat{\Hom(G_n,{\Bbb Z}_2)}]$ (or equivalently,
$\alpha(E_v)=\{\alpha(x)| x\in E_v\}$ forms a basis of
$\Hom(G_n,{\Bbb Z}_2)$), where $E_v$ denotes the set of all edges
adjacent to $v$;

\item[(P2)] for each edge $e$ of $\Gamma_M$, $\alpha(E_u)\equiv
\alpha(E_v) \mod \alpha(e)$ in $\Hom(G_n,{\Bbb Z}_2)$ where $u$ and
$v$ are two endpoints of $e$.
\end{enumerate}
The pair $(\Gamma_M, \alpha)$ is called the {\em $G_n$-colored
graph} of the 2-torus manifold $M^n$ here.

\begin{rem}
 The property (P2) has the following
equivalent statement that for each edge $e=uv$ of $\Gamma_M$, there
is a unique bijection $\theta_{e}: E_u\longrightarrow E_v$ such that
for any $e'\in E_u\backslash\{e\}$,
$$\alpha(e')\equiv\alpha(\theta_{e}(e'))\mod \alpha(e).$$
The collection $\theta=\{\theta_e\big\vert e\in E_{\Gamma_M}\}$ is
called a \emph{connection} of $(\Gamma_M, \alpha)$. Geometrically,
 $\{\prod_{x\in E_v}\alpha(x)|v\in V_{\Gamma_M}\}$ or $\{\alpha(E_v)| v\in V_{\Gamma_M}\}$ means
 the collection of all tangential $G_n$-representations at fixed points in
 $M^n$, and $\theta=\{\theta_e\big\vert e\in E_{\Gamma_M}\}$ means the
 connection among all tangential $G_n$-representations at fixed
 points.
\end{rem}

\begin{example}
Consider the $n$-dimensional real projective space ${\Bbb R}P^n$
($n\geq 2$) with the standard linear $G_n$-action
        defined by $$[x_0,x_1,...,x_n]\longmapsto[x_0,(-1)^{g_1}x_1,...,(-1)^{g_n}x_n]$$
     which fixes $n+1$ isolated points $[0,...,0,1,0,...,0]$ with 1
in the $i$-th place for $i=0,1,...,n$. This action determines a
unique regular graph $\Gamma_{{\Bbb R}P^n}$, which is just the
1-skeleton of an $n$-simplex $\Delta^n$, and the ${{n+1}\choose 2}$
edges of $\Gamma_{{\Bbb R}P^n}$ are colored by $\rho_1, ..., \rho_n,
\rho_i+\rho_j, 1\leq i<j\leq n$, respectively, where $\{\rho_1, ...,
\rho_n\}$ is the standard basis of $\mbox{Hom}(G_n,{\Bbb Z}_2)$.
When $n=2, 3$, the $G_n$-colored graph $\Gamma_{{\Bbb R}P^n}$ is
shown in Figure~\ref{n2}.
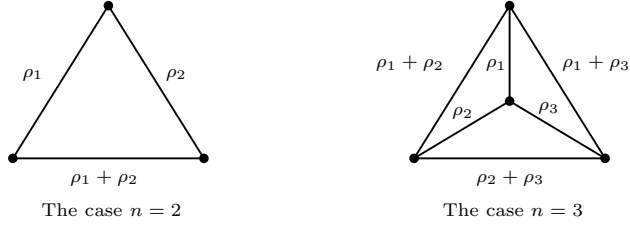
\begin{figure*}[h]
    \input{n2.pstex_t}\centering
    \caption[a]{Colored graphs for the cases $n=2, 3$. }\label{n2}
\end{figure*}

We note that the diagonal action on two copies of the standard $(\Z_2)^2$-action on ${\Bbb
R}P^2$ and the twist involution on the product ${\Bbb R}P^2\times{\Bbb R}P^2$ may give a
$({\Bbb Z}_2)^3$-action on ${\Bbb R}P^2\times{\Bbb R}P^2$ fixing
three fixed points. However, ${\Bbb R}P^2\times{\Bbb R}P^2$ with this $({\Bbb Z}_2)^3$-action is not a 2-torus manifold, but it can determine a colored graph, as shown in Figure~\ref{n3}
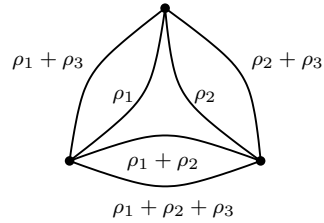
\begin{figure*}[h]
    \input{n3.pstex_t}\centering
    \caption[a]{The colored graph of the $({\Bbb Z}_2)^3$-action on ${\Bbb R}P^2\times{\Bbb R}P^2$. }\label{n3}
\end{figure*}
\end{example}

Guillemin and Zara \cite{gz} formulated the results of GKM theory in
terms of a colored graph, and developed the GKM theory
combinatorially. They defined and studied the abstract GKM graphs.
This idea may still be carried out in the mod 2 case.  Following
\cite{l2}, let $\Gamma$ be a finite regular graph of valence $n$
without loops. If there is a map $\alpha$ from the set $E_\Gamma$ of
all edges of $\Gamma$ to all nontrivial elements of $\Hom(G_n,{\Bbb
Z}_2)$ satisfying  the properties (P1) and (P2) as above,  then the
pair $(\Gamma, \alpha)$ is called an  {\em abstract $G_n$-colored
graph} of $\Gamma$, and $\alpha$ is called a {\em $G_n$-coloring} on
$\Gamma$.

\begin{rem}
By definition, it is easy to see that if all $\alpha(E_v), v\in V_{\Gamma}$, are
distinct, then for each edge $e\in E_{\Gamma}$, $|E_e|=1$ where
$V_{\Gamma}$ denotes the set  of vertices in $\Gamma$ and $E_e$
denotes all edges joining two endpoints of $e$ (see also \cite[Lemma
5.1]{l3}).
\end{rem}

Obviously, an abstract $G_n$-colored graph $(\Gamma, \alpha)$
determines a faithful $G_n$-polynomial $\sum_{v\in V_\Gamma}
\prod_{x\in E_v}\alpha(x)$ in $\Z_2[\widehat{\Hom(G_n,{\Bbb
Z}_2)}]$.

\begin{defn} Let $(\Gamma, \alpha)$ be an abstract $G_n$-colored graph.
Set $$g_{(\Gamma, \alpha)} = \sum_{v\in V_\Gamma} \prod_{x\in
E_v}\alpha(x)$$ which is called the {\em $G_n$-coloring polynomial}
of $(\Gamma, \alpha)$.
\end{defn}

It was shown in \cite[Proposition 2.2]{l2} that for an  abstract
$G_n$-colored graph $(\Gamma, \alpha)$, the collection
$\{\alpha(E_v), v\in V_{\Gamma}\}$ is always realizable as the fixed
point data of some 2-torus manifold $M^n$, which implies that the
$G_n$-coloring polynomial $g_{(\Gamma, \alpha)}$ of $(\Gamma,
\alpha)$ must belong to the image of $\phi_n:
\mathcal{Z}_n(G_n)\longrightarrow\mathcal{R}_n(G_n)$, where
$\mathcal{R}_n(G_n)$ is regarded as the subgroup of
$\Z_2[\widehat{\Hom(G_n, \Z_2)}]$. However, this result does not
tell us whether $(\Gamma, \alpha)$ is the $G_n$-colored graph
$(\Gamma_M, \alpha)$ of $M^n$ or not, which is related  to the
following geometric realization problem: {\em under what condition
can $(\Gamma, \alpha)$ become a $G_n$-colored graph of some 2-torus
manifold?} Some work for the geometric realization problem has been
studied in details in \cite{bl}.

\vskip .1cm

On the other hand, as mentioned above, we have known  from \cite{bl}
or \cite[Section 2]{l3} that each 2-torus manifold $M^n$  determines
a $G_n$-colored graph $(\Gamma_M, \alpha)$, and the corresponding
$G_n$-coloring polynomial $g_{(\Gamma_M, \alpha)}$ is exactly
$\phi_n(\{M^n\})$.  Since $\phi_n$ is a monomorphism and the
$G_n$-colored graph $(\Gamma_M, \alpha)$  may naturally be
understood as an abstract $G_n$-colored graph, it follows that
\begin{thm}\label{color poly}
A faithful $G_n$-polynomial $g\in \Z_2[\widehat{\Hom(G_n,{\Bbb
Z}_2)}]$ belongs to $\IM\phi_n$ if and only if it is the
$G_n$-coloring polynomial of an abstract $G_n$-colored graph
$(\Gamma, \alpha)$.
\end{thm}

Now by $\Lambda(G_n)$ we denote the set of all abstract
$G_n$-colored graphs $(\Gamma, \alpha)$.

\begin{defn}
Two abstract $G_n$-colored graphs $(\Gamma_1, \alpha_1)$ and
$(\Gamma_2, \alpha_2)$ in $\Lambda(G_n)$ are said to be {\em
equivalent} if $g_{(\Gamma_1, \alpha_1)}=g_{(\Gamma_2, \alpha_2)}$,
denoted by $(\Gamma_1, \alpha_1)\sim(\Gamma_2, \alpha_2)$.
\end{defn}

On the coset $\Lambda(G_n)/\sim$, define the addition $+$ as
follows:
$$\{(\Gamma_1, \alpha_1)\}+\{(\Gamma_2, \alpha_2)\}:=\{(\Gamma_1, \alpha_1)\sqcup(\Gamma_2, \alpha_2)\}$$
where $\sqcup$ means the disjoint union.  Then $\Lambda(G_n)/\sim$
forms an abelian group, where the zero element in
$\Lambda(G_n)/\sim$ is the class of the abstract $G_n$-colored graph
with zero $G_n$-coloring polynomial. By Theorem~\ref{color poly} we
have that
\begin{cor}\label{graph}
$\mathcal{Z}_n(G_n)$ is isomorphic to $\Lambda(G_n)/\sim$.
\end{cor}

\begin{defn}\label{prime graph}
An abstract $G_n$-colored graph $(\Gamma, \alpha)$ in $\Lambda(G_n)$
with $g_{(\Gamma, \alpha)}\not=0$ is said to be {\em prime} if all
$\alpha(E_v), v\in V_\Gamma$,  are distinct.
\end{defn}
It is easy to see that a prime abstract $G_n$-colored graph
$(\Gamma, \alpha)$ has the property that $|V_\Gamma|$ equals to the
number of monomials of $g_{(\Gamma, \alpha)}$. Now let us look at
nonzero classes in $\Lambda(G_n)/\sim$.

\begin{lem}\label{prime}
Each nonzero class of $\Lambda(G_n)/\sim$ contains a prime abstract
$G_n$-colored graph as its representative.
\end{lem}
\begin{proof}
Let $(\Gamma, \alpha)$ be an abstract $G_n$-colored graph in
$\Lambda(G_n)$ with $g_{(\Gamma, \alpha)}\not=0$. Suppose that
$(\Gamma, \alpha)$ is not prime. Then there must be two vertices $v$
and $u$ such that $\alpha(E_v)=\alpha(E_u)$. Now let us perform a
``connected sum'' of $(\Gamma, \alpha)$ to itself at $v$ and $u$ as
follows: first cut out two vertices $v$ and $u$, and then glue $n$
edges $\{e_v^1, ..., e_v^n\}$ removed $v$ to $n$ edges $\{e_u^1,
..., e_u^n\}$ removed $u$  along sectional endpoints respectively in
such a way that two $e_v^i$ and $e_u^j$ will be glued together as
long as $\alpha(e_v^i)=\alpha(e_u^j)$. Then it is easy to see that
the resulting graph $(\Gamma', \alpha')$ is still an abstract
$G_n$-colored graph in $\Lambda(G_n)$. This decreases two vertices
$v$ and $u$ from $(\Gamma, \alpha)$, and clearly $(\Gamma',
\alpha')\sim (\Gamma, \alpha)$. Since $\Gamma$ is finite, this
procedure can be ended until we obtain the desired prime abstract
$G_n$-colored graph.
\end{proof}

\begin{rem}
 For two abstract $G_n$-colored graphs $(\Gamma_1, \alpha_1)$ and $(\Gamma_2, \alpha_2)$,
 if there are two vertices $v_1\in V_{\Gamma_1}$ and $v_2\in V_{\Gamma_2}$ such that
$\alpha_1(E_{v_1})=\alpha_2(E_{v_2})$, in a similar way as shown   in the proof of Lemma~\ref{prime},
we can define a connected sum $(\Gamma_1, \alpha_1)\sharp_{v_1, v_2}(\Gamma_2, \alpha_2)$
of $(\Gamma_1, \alpha_1)$ and $(\Gamma_2, \alpha_2)$ at $v_1$ and $v_2$. Then we can obtain that
$$ \{(\Gamma_1, \alpha_1)\sharp_{v_1, v_2}(\Gamma_2, \alpha_2)\}=
\{(\Gamma_1, \alpha_1)\}+\{(\Gamma_2, \alpha_2)\}=\{(\Gamma_1,
\alpha_1)\sqcup(\Gamma_2, \alpha_2)\}.$$ This also implies that for
$\{M_1\}, \{M_2\}\in \mathcal{Z}_n(G_n)$, if there are two fixed
points $p_1\in M_1^{G_n}$ and $p_2\in M_2^{G_n}$ such that the
tangent $G_n$-representations at $p_1$ and $p_2$ are isomorphic,
then we can perform an equivariant connected sum $M_1\sharp_{p_1,
p_2}M_2$ of $M_1$ and $M_2$ at $p_1$ and $p_2$, and in particular,
$$\{M_1\sharp_{p_1, p_2}M_2\}=\{M_1\}+\{M_2\}=\{M_1\sqcup M_2\}.$$
\end{rem}

\subsection{Small covers}\label{small}
 In \cite{dj} Davis and Januszkiewicz introduced and studied the topological version of
real toric variety, i.e.,  ``small cover''. This gives another link between the equivariant topology
and the combinatorics of simple convex polytopes.

\vskip .1cm An $n$-dimensional {\em small cover} $\pi:
M^n\longrightarrow P^n$ is a smooth closed $n$-manifold $M^n$ with a
locally standard $G_n$-action such that its orbit space is
homeomorphic to a simple convex $n$-polytope $P^n$, where a locally
standard $G_n$-action on $M^n$ means that this $G_n$-action on $M^n$
is locally isomorphic to a faithful representation of $G_n$ on
${\Bbb R}^n$. A small cover is a special 2-torus manifold. Each
small cover $\pi: M^n \longrightarrow P^n$ determines a
characteristic function $\lambda$ (here we call it a {\em
$G_n$-coloring}) on $P^n$, defined by mapping all facets (i.e.,
$(n-1)$-dimensional faces) of $P^n$ to nontrivial elements of
$\Hom(\Z_2, G_n)$ such that $n$ facets meeting at each vertex are
mapped to $n$ linearly independent elements. There are many
fascinating characteristics for $\pi: M^n\longrightarrow P^n$,
saying that the algebraic topology of $M^n$ is essentially
consistent with the algebraic combinatorics of $(P^n, \lambda)$ in
many aspects, and $M^n$ can be recovered by the pair $(P^n,
\lambda)$. For example, the mod 2 Betti numbers $(b_0, b_1, ...,
b_n)$ of $M^n$ agree with the $h$-vector $(h_0, h_1, ..., h_n)$ of
$P^n$. This leads us to one of reasons why we posed the Conjecture
$(\star)$ in \cite{l1}. The other one is that in \cite{br}
Buchstaber and Ray gave  the proof of the Conjecture $(\star)$ in
non-equivariant case, i.e.,  each $n$-dimensional class of
$\mathfrak{N}_*$ contains a small cover  as its representative,
where $\mathfrak{N}_*$ denotes the Thom unoriented borism ring.

 \vskip .2cm

 Now suppose that $\pi: M^n\longrightarrow P^n$ is a small cover, and
 $\lambda: \mathcal{F}(P^n)\longrightarrow
 \Hom(\Z_2, G_n)$ is its characteristic function, where $\mathcal{F}(P^n)$ consists of all facets of $P^n$.
Given a vertex $v$ of $P^n$, since $P^n$ is simple, there are $n$
facets $F_1, ..., F_n$ in $\mathcal{F}(P^n)$ such that
$v=F_1\cap\cdots\cap F_n$. Then, by the definition of $\lambda$, we
see  that the vertex $v$  determines a monomial
$\prod_{i=1}^n\lambda(F_i)$  of degree $n$ in
$\Z_2[\widehat{\Hom(\Z_2, G_n)}]$, whose dual by the pairing
(\ref{pairing}) is faithful in $\Z_2[\widehat{\Hom(G_n, \Z_2)}]$.
Here $\prod_{i=1}^n\lambda(F_i)$ is called the {\em $G_n$-coloring
monomial at $v$},  denoted by $\lambda_v$.  Moreover, all vertices
in the vertex set $V_{P^n}$ of $P^n$ via $\lambda$ give a polynomial
  $\sum_{v\in V_{P^n}}\lambda_v$ of degree $n$ in $\Z_2[\widehat{\Hom(\Z_2, G_n)}]$, which is denoted by
 $g_{(P^n, \lambda)}$, and is called  $g_{(P^n, \lambda)}$ the {\em $G_n$-coloring polynomial}
 of $(P^n, \lambda)$.

\begin{rem}
Geometrically, each degree-one factor $\lambda(F_i)$ of the monomial
$\lambda_v$ at the vertex $v=F_1\cap\cdots\cap F_n$ is actually the
normal representation to the characteristic submanifold
$\pi^{-1}(F_i)$ fixed by the $\Z_2$-subgroup of $G_n$ corresponding
to the factor $\lambda(F_i)$. Then, the fixed point $\pi^{-1}(v)$ is
in the intersection of those $n$ characteristic submainfolds
$\pi^{-1}(F_i), i=1, ..., n$, determined by $n$ degree-one factors
of $\lambda_v$. This means that the dual of $\lambda_v$ by the
pairing (\ref{pairing}) is exactly the tangential
$G_n$-representation at the fixed point $\pi^{-1}(v)$ (see also
\cite[Proposition 4.1]{l2}). This is also shown in \cite[Lemmas 5.48
and 5.50]{bp} for the case of quasi-toric manifolds in terms of
matrices.
\end{rem}

 Now let $(\Gamma_M, \alpha)$ be the $G_n$-colored graph of $\pi:
M^n\longrightarrow P^n$, and let $g_{(\Gamma_M, \alpha)}$ be the
$G_n$-coloring polynomial of $(\Gamma_M, \alpha)$. We know from
\cite[Proposition 4.1; Remark 4]{l2} that $\Gamma_M$ is exactly the
1-skeleton of $P^n$, and both $\lambda$ and $\alpha$ determine each
other. Specifically, let $v$ be a vertex of $P^n$. 
Then there are $n$ facets $F_1, ..., F_n$ in $P^n$ and $n$ edges (or
1-faces) $e_1, ..., e_n$ in $\Gamma_M$ such that
$v=F_1\cap\cdots\cap F_n= e_1\cap\cdots \cap e_n$. With no loss of
generality, assume that $e_i$ is the intersection of $F_1, ...,
F_{i-1}, F_{i+1}, ..., F_n$, so $F_i$ contains $n-1$ edges $e_1,
..., e_{i-1}, e_{i+1}, ..., e_n$ except for $e_i$. Then both
$\lambda$ and $\alpha$ are dual in the following sense
$$\langle\lambda(F_i), \alpha(e_j)\rangle=\begin{cases}
1 & \text{ if } i=j\\
0 & \text{ if } i\not= j
\end{cases}$$
which implies that the basis $\{\lambda(F_1), ..., \lambda(F_n)\}$
of $\Hom({\Bbb Z}_2, G_n)$ is the dual basis of $\{\alpha(e_1), ...,
\alpha(e_n)\}$ in $\Hom(G_n, {\Bbb Z}_2)$. This gives
\begin{prop}\label{small}
$g_{(P^n, \lambda)}$ is the dual polynomial of $g_{(\Gamma_M,
\alpha)}$.
\end{prop}

\begin{rem}\label{skeleton}
 Given an automorphism $\sigma$ of $\Hom(\Z_2, G_n)$,  we can induce
 a new $G_n$-coloring $\sigma\circ \lambda$ on $P^n$ from $\lambda: \mathcal{F}(P^n)\longrightarrow
 \Hom(\Z_2, G_n)$.
Furthermore, we may see easily that $g_{(P^n, \sigma\circ
\lambda)}=\sigma(g_{(P^n,\lambda)})$.
\end{rem}

\begin{example} \label{exam-basis}
Regard $S^1$ as the unit circle $\{ z\in {\Bbb C}\ \big| |z|=1 \}$
in ${\Bbb C}$ and ${\Bbb R} P^2$ as the projective plane ${\Bbb R}
P({\Bbb C}\oplus {\Bbb R}) = \{\, [v,w]\big| v\in {\Bbb C}, w\in
{\Bbb R}\,
      \}$ in ${\Bbb C}\oplus {\Bbb R}$, we then construct
a $(\mathbb{Z}_2)^3$-action on $M^3=S^1\times \mathbb{R}P^2$
defined by the following three commutative involutions
 \begin{align*}
           t_1: (z,[v,w]) &\longmapsto (\bar{z}, [zv,w]) \\
            t_2: (z,[v,w]) &\longmapsto (z, [-\bar{z}\bar{v},w])\\
             t_3: (z,[v,w]) &\longmapsto (\bar{z}, [-zv,w]).
        \end{align*}
We know from \cite[Lemma 4.3]{ly} that $M^3$ is a 3-dimensional small cover whose orbit is a 3-sided prism
$P^3(3)$ with a $({\Bbb Z}_2)^3$-coloring $\lambda$, and its colored graph $(\Gamma_M, \alpha)$ is the 1-skeleton of $P^3(3)$,  shown as follows:
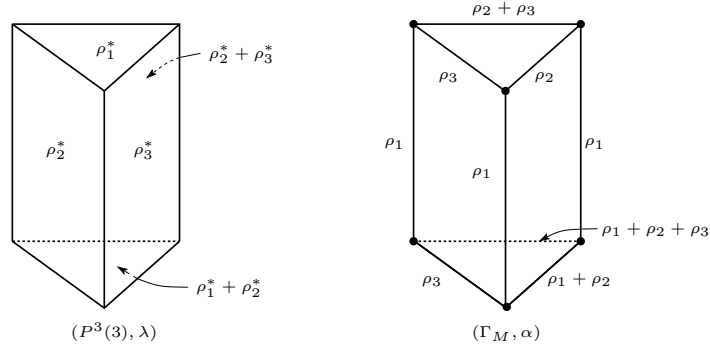
\[\input{l.pstex_t}\centering   \]
   The corresponding $G_3$-coloring polynomials $g_{(\Gamma_M, \alpha)}$ and $g_{(P^3(3), \lambda)}$ are
   exactly the faithful polynomial $g$ and its dual polynomial $g^*$ respectively, as expressed in
   Example~\ref{example-dual}.
\end{example}

Generalized real Bott manifolds belong to a class of nicely behaved
small covers, which were introduced  and studied in \cite{cmo, cms,
cms2, km, m}. A {\em generalized real Bott tower} of height $n$ is a
sequence of ${\Bbb R}P^{n_i}$-bundles with $n_i\geq 1$:
\[
\begin{CD} B^{\Bbb R}_n@ >{\pi_n}>> B^{\Bbb R}_{n-1}@>{\pi_{n-1}}>> \cdots@>{\pi_2}>>B^{\Bbb R}_{1}@
>{\pi_1}>>B^{\Bbb R}_{0}=\{\text{a point}\}
\end{CD}
\]
where each $\pi_i: B^{\Bbb R}_i\longrightarrow B^{\Bbb R}_{i-1}$ for $i=1, ..., n$ is the projectivization of a Whitney sum of $n_i+1$ real line bundles over $B^{\Bbb R}_i$.  We call $B^{\Bbb R}_i$ an {\em $i$-stage generalized real Bott
manifold} or a {\em generalized real Bott manifold of height $i$}.  It is easy to check that  the $i$-stage generalized real Bott
manifold $B^{\Bbb R}_i$ is a small cover over $\Delta^{n_1}\times\cdots\times\Delta^{n_i}$ where $\Delta^{n_j}$ is an $n_j$-dimensional simplex. Conversely, it was shown in \cite{cms} that a small cover over a product of simplices is  a generalized real Bott manifold.

\vskip .1cm
In the special case where $n_i=1$ for all $i$,  $B^{\Bbb R}_i$ is called an {\em $i$-stage real Bott manifold}, and it is a small cover over an $i$-cube.

\vskip .1cm Much interesting work related to generalized real Bott
manifolds has been carried on (see~\cite{cmo, cms, cms2, km, m, y}).
For example, it was proved in \cite{km} that the cohomological
rigidity for real Bott manifolds holds, and we also know from
\cite{cmo} that real Bott manifolds provide examples of flat
riemannian manifolds, and they are also related to acyclic digraphs.

\subsection{$G_n$-colorings on the product of simple convex
polytopes}\label{product-coloring}

We note that the natural identification of $G_n$ with  $\Hom(\Z_2, G_n)$ gives a
correspondence $\Theta$ between their subgroups. 

\vskip .1cm


The following is a generalization of $G_n$-colorings on simple
convex $n$-polytopes.
\begin{defn}\label{coloring-gen}
Let $P^k$ be a simple convex $k$-polytope with $k\leq n$. A {\em
$G_n$-coloring} $\lambda$ on $P^k$ is a map from all facets of $P^k$
to  $\Hom(\Z_2, G_n)$ such that $\lambda$ maps $k$ facets at each
vertex of $P^k$ into $k$ linearly independent elements in
$\Hom(\Z_2, G_n)$.  In particular, a $G_n$-coloring $\lambda$ on
$P^k$ is said to be {\em nice} if there is a subgroup $L_\lambda$ of
rank $n-k$ in $\Hom(\Z_2, G_n)$ such that
$\overline{\lambda}=\ell\circ \lambda$ maps $k$ facets at each
vertex of $P^k$ into $k$ linearly independent elements in the
quotient group $\Hom(\Z_2, G_n)/L_\lambda$, where $\ell$ denotes the
natural quotient map from $\Hom(\Z_2, G_n)$ to $\Hom(\Z_2,
G_n)/L_\lambda$.
\end{defn}

Clearly, each $G_n$-coloring $\lambda$ on $P^k$ can still determine
a polynomial as before, which is also denoted by $g_{(P^k,
\lambda)}$, and is called the $G_n$-coloring polynomial of $(P^k,
\lambda)$. \vskip .2cm

For a nice $G_n$-coloring $\lambda$ on $P^k$, we can choose a
subgroup $C_\lambda$ of rank $k$ in $\Hom(\Z_2, G_n)$ such that
$C_\lambda\cong \Hom(\Z_2, G_n)/L_\lambda$ and  $L_\lambda\oplus
C_\lambda=\Hom(\Z_2, G_n)$.
 Write $G'=\Theta^{-1}(C_\lambda)$, then
$\overline{\lambda}$ is actually a $G'$-coloring on $P^k$, so it is also
regarded as a $(\Z_2)^k$-coloring on $P^k$ since $G'$ is isomorphic
to $(\Z_2)^k$.
 Since $\Hom(\Z_2, (\Z_2)^k)\subseteq \Hom(\Z_2, G_n)$ for
$k\leq n$,
 each $(\Z_2)^k$-coloring on $P^k$  can always be regarded as a nice $G_n$-coloring on $P^k$.

\vskip .2cm

Let $P^n$ be the product $P_1^{n_1}\times P_2^{n_2}$ of two simple
convex polytopes $P_1^{n_1}$ and  $P_2^{n_2}$ with $n_1+n_2=n$.
Suppose that  $P^n$ admits a $G_n$-coloring $\lambda$. Then for
$i=1,2$,  each $P_i^{n_i}$ naturally inherits a $G_n$-coloring
$\lambda_i$ in such a way that for each facet $F$ of $P_i^{n_i}$,
\begin{equation}\label{def1}
\lambda_i(F)=
\begin{cases}
\lambda(F\times P_2^{n_2}) & \text{ if $F$ is a facet of $P_1^{n_1}$}\\
\lambda(P_1^{n_1}\times F) & \text{ if $F$ is a facet of
$P_2^{n_2}$}
\end{cases}
\end{equation}
which is called the {\em restriction} to $P_i^{n_i}$ of $\lambda$.
Note that all facets of $P_1^{n_1}\times P_2^{n_2}$ consist of the
polytopes of the forms $F_1\times P_2^{n_2}$ and $P_1^{n_1}\times
F_2$ where $F_i$ is a facet of $P_i^{n_i}$.

\begin{lem}\label{pro}
For $i=1,2$,  the $G_n$-coloring $\lambda_i$ on $P^{n_i}_i$ is nice.
\end{lem}

\begin{proof}
Let $v_i$ ($i=1,2$) be a vertex of $P_i^{n_i}$, and let $F_1^{(i)},
..., F_{n_i}^{(i)}$ be $n_i$ facets adjacent to $v_i$ in
$P_i^{n_i}$. Then $F_1^{(1)}\times P_2^{n_2}, ...,
F_{n_1}^{(1)}\times P_2^{n_2}, P_1^{n_1}\times F_1^{(2)}, ...,
P_1^{n_1} \times F_{n_2}^{(2)}$ are $n$ facets at vertex $(v_1,
v_2)$ of $P^n$. So $\lambda_1(F_1^{(1)}), ...,
\lambda_1(F_{n_1}^{(1)}), \lambda_2(F_1^{(2)}), ...,
\lambda_2(F_{n_2}^{(2)})$ form a basis of $\H(\Z_2, G_n)$. When
$v_1$ runs over all vertices of $P_1^{n_1}$ and $v_2$ is fixed,
since each $(v_1, v_2)$ is always a vertex of $P^n$,  take
$L_{\lambda_1}$ as $\text{Span}\{\lambda_2(F_1^{(2)}), ...,
\lambda_2(F_{n_2}^{(2)})\}$, we see easily that
$\overline{\lambda_1}$ maps $n_1$ facets at each vertex of $P^{n_1}$
into $n_1$ linearly independent elements in $\H(\Z_2,
G_n)/L_{\lambda_1}$. So $\lambda_1$ is a nice $G_n$-coloring. In a
similar way, we have that $\lambda_2$ is a nice $G_n$-coloring, too.
\end{proof}

\begin{rem}
   By the proof of Lemma~\ref{pro}, we can choose  $L_{\lambda_1}$ and $L_{\lambda_2}$
    via the $G_n$-coloring at a vertex $(v_1, v_2)$ of $P^n$ such that
    $L_{\lambda_1}\oplus L_{\lambda_2}=\H(\Z_2, G_n)$. In fact, let $F_1^{(1)}\times P_2^{n_2}, ...,
F_{n_1}^{(1)}\times P_2^{n_2}, P_1^{n_1}\times F_1^{(2)}, ...,
P_1^{n_1} \times F_{n_2}^{(2)}$ be $n$ facets meeting at $(v_1,
v_2)$. Then we can choose $L_{\lambda_i}$ as $\text{Span}\{\lambda_i(F_1^{(i)}), ...,
\lambda_i(F_{n_i}^{(i)})\}, i=1,2$, as desired.
\end{rem}

\begin{prop}[Product formula]\label{formula}
$$g_{(P_1^{n_1}\times P_2^{n_2}, \lambda)}=g_{(P_1^{n_1}, \lambda_1)}g_{(P_2^{n_2},
\lambda_2)}.$$
\end{prop}

\begin{proof}
First we see  from the proof of Lemma~\ref{pro} that
$$g_{(P_1^{n_1}\times P_2^{n_2}, \lambda)}=\sum_{(v_1, v_2)\in V_{P_1^{n_1}\times
P_2^{n_2}}}\lambda_{(v_1,v_2)}=\sum_{v_2\in
V_{P^{n_2}_2}}\sum_{v_1\in V_{P^{n_1}_1}}\lambda_{(v_1,v_2)}$$ where
$V_P$ denotes the vertex set of a simple convex polytope $P$. Since
$\lambda_{(v_1,v_2)}=\lambda_{1v_1}\lambda_{2v_2}$, we have that
\begin{align*}g_{(P_1^{n_1}\times P_2^{n_2}, \lambda)}
&=\sum_{v_2\in V_{P^{n_2}_2}}\sum_{v_1\in
V_{P^{n_1}_1}}\lambda_{1v_1}\lambda_{2v_2}\\
&=\Big(\sum_{v_2\in
V_{P^{n_2}_2}}\lambda_{2v_2}\Big)\Big(\sum_{v_1\in
V_{P^{n_1}_1}}\lambda_{1v_1}\Big)\\
&=g_{(P_1^{n_1}, \lambda_1)}g_{(P_2^{n_2}, \lambda_2)} \end{align*}
as desired.
\end{proof}

Proposition~\ref{formula} also implies the following result.

\begin{cor}\label{indecom}
 Suppose that a simple convex $n$-polytope $P^n$ admits a $G_n$-coloring $\lambda$
 such that $g_{(P^n, \lambda)}$ is indecomposable in $\Z_2[\widehat{\Hom(\Z_2, G_n)}]$.
 Then $P^n$ is indecomposable too.
\end{cor}

Now for a decomposable simple convex $n$-polytope
$P^n=P_1^{n_1}\times P_2^{n_2}$ with $n=n_1+n_2$. Suppose that
$P_i^{n_i} (i=1,2)$ admits a nice $G_n$-coloring $\lambda_i$ such
that there are two subgroups $L_{\lambda_1}$ and $L_{\lambda_2}$
corresponding to $\lambda_1$ and $\lambda_2$ respectively which
satisfy $L_{\lambda_1}\oplus L_{\lambda_2}=\H(\Z_2, G_n)$. Then we
can define a $G_n$-coloring $\lambda$ on $P^n$ as follows:
  for each facet $F_i$ of $P_i^{n_i}$,
\begin{equation}\label{def2}
\lambda(F)=
\begin{cases}
\lambda_1(F_1) & \text{if } F=F_1\times P_2^{n_2} \\
 \lambda_2(F_2) & \text{if } F=P_1^{n_1}\times F_2.
\end{cases}\end{equation}
  Furthermore, it is easy to see
that
 the product formula in
Proposition~\ref{formula} still holds in this case. Namely,
$g_{(P^n, \lambda)}=g_{(P_1, \lambda_1)}g_{(P_2, \lambda_2)}$.

\begin{rem}\label{pr}
By the definitions of $\lambda_i$ and $\lambda$ in (~\ref{def1}) and
(\ref{def2}), we see that  $P^n=P_1^{n_1}\times P_2^{n_2}$ with a
$G_n$-coloring $\lambda$ uniquely defines the restrictions
$\lambda_i$ of $\lambda$ on $P_i^{n_i}$. Conversely, given two nice
$G_n$-colorings $\lambda_i$ on $P_i^{n_i}$ ($i=1,2$) with
$n_1+n_2=n$ such that there are two corresponding subgroups
$L_{\lambda_1}$ and $L_{\lambda_2}$ with $L_{\lambda_1}\oplus
L_{\lambda_2}=\H(\Z_2, G_n)$, we can uniquely define a
$G_n$-coloring $\lambda$ on $P^n=P_1^{n_1}\times P_2^{n_2}$ such
that $\lambda_i$ is the restriction to $P^{n_i}_i$ of $\lambda$.
With this understood, we shall write
$$(P^n, \lambda)=(P^{n_1}, \lambda_1)\times (P^{n_2}, \lambda_2)$$
if $\lambda_i$ is the restriction to $P^{n_i}_i$ of $\lambda$.
\end{rem}

Throughout the following we use the convention that all simple convex $n$-polytopes are embedded in ${\Bbb R}^n$, and if  two simple convex polytopes $P^n_1$ and $P^n_2$ are combinatorially equivalent, then $P^n_1$ is identified with $P^n_2$.

\vskip .1cm

Now suppose $(P_1^n, \lambda_1)$ and $(P_2^n, \lambda_2)$ are two
$G_n$-colored simple convex $n$-polytopes such that there are two
vertices $v_1\in P_1^n$ and $v_2\in P_2^n$ with
$\lambda_{1v_1}=\lambda_{2v_2}$. Then  we can always perform a
connected sum $P_1^n\sharp_{v_1, v_2} P_2^n$ of $(P_1^n, \lambda_1)$
and $(P_2^n, \lambda_2)$ at $v_1$ and $v_2$, so that
$P_1^n\sharp_{v_1, v_2} P_2^n$ is still a simple convex polytope. In
fact,  because $(P^n_i, \lambda_i)$ is identified with its mirror
reflection $(\overline{P}^n_i, \overline{\lambda}_i)$ along a
hyperplane in ${\Bbb R}^n$ but the $G_n$-coloring order of $n$
facets at $v_i$ in $(P^n_i, \lambda_i)$ is exactly the reversion of
the $G_n$-coloring order of $n$ facets at $\overline{v}_i$ in
$(\overline{P}^n_i, \overline{\lambda}_i)$ where $\overline{v}_i$ is
the reflection point of $v_i$, this means that we can choose
$(P_1^n, \lambda_1)$ and $(P_2^n, \lambda_2)$  up to combinatorial
equivalence (if necessary) such that the $G_n$-coloring order of $n$
facets at $v_1$ in $(P^n_1, \lambda_1)$ is the reversion of the
$G_n$-coloring order of $n$ facets at $v_2$ in $(P^n_2, \lambda_2)$.
Note that the coloring monomial at a vertex $v$ determines  the
$G_n$-coloring order of the facets adjacent to $v$ up to a
reflection since the polytope is simple. Thus, we can perform the
required connected sum between $(P^n_1, \lambda_1)$ and $(P^n_2,
\lambda_2)$. Next it is not difficult to see that $\lambda_1$ and
$\lambda_2$ determine a $G_n$-coloring $\lambda$ on
$P_1^n\sharp_{v_1, v_2} P_2^n$. Moreover, we have that

\begin{prop}[Connected sum formula]\label{sum formula}
$$g_{(P_1^n\sharp_{v_1, v_2} P_2^n, \lambda)}=g_{(P_1^n, \lambda_1)}+g_{(P_2^n,
\lambda_2)}.$$
\end{prop}

\begin{proof}
This is because $g_{(P_1^n\sharp_{v_1, v_2} P_2^n, \lambda)}=
(g_{(P_1^n, \lambda_1)}-\lambda_{1v_1})+(g_{(P_2^n,
\lambda_2)}-\lambda_{2v_2}) =g_{(P_1^n, \lambda_1)}+g_{(P_2^n,
\lambda_2)}$ in $\Z_2[\widehat{\Hom(\Z_2, G_n)}]$.
\end{proof}

\section{Proof of Theorem~\ref{main result}} \label{proof of main}

First let us state a useful lemma.

\begin{lem}\label{local}
If $t\eta_2\cdots\eta_n$ and $t\theta_2\cdots\theta_n$ are two
different faithful $G_n$-monomials of degree $n$ in
$\Z_2[\widehat{\Hom(G_n,\Z_2)}]$ with $\{\eta_2, ..., \eta_n\}\equiv
\{\theta_2, ..., \theta_n\}\mod t$ in $\Hom(G_n,\Z_2)$, then their
duals are different and contain  the same monomial  of degree $n-1$
as a factor. The converse is also true.
\end{lem}

\begin{proof}
This  follows immediately from the pairing (\ref{pairing}).
\end{proof}

\begin{prop}\label{p1}
Let $g$ be a faithful $G_n$-polynomial in $\Z_2[\widehat{\Hom(G_n,
\Z_2)}]$. If $g\in \IM\phi_n$ then $d(g^*)=0$.
\end{prop}

\begin{proof}
Let $g\in \IM\phi_n$. Then by Theorem~\ref{color poly},
Corollary~\ref{graph} and Lemma~\ref{prime}, there is a prime
abstract $G_n$-colored graph $(\Gamma, \alpha)$ with $g_{(\Gamma,
\alpha)}=g$. Take an edge $e=pq$ in $\Gamma$, we have that
$\alpha(E_p)\setminus\{\alpha(e)\}\equiv
\alpha(E_q)\setminus\{\alpha(e)\}\mod \alpha(e)$ in $\Hom(G_n,
\Z_2)$. Moreover,  by Lemma~\ref{local}, it follows that all
monomials of degree $n-1$ in $d(g^*)$ appear in pairs, so
$d(g^*)=0$.
\end{proof}

\begin{prop}\label{p2}
Suppose that $g=\sum_{i=1}^\ell t_{i,1}\cdots t_{i, n}$ is a
faithful $G_n$-polynomial in $\Z_2[\widehat{\Hom(G_n, \Z_2)}]$. If
$d(g^*)=0$, then $g$ is the $G_n$-coloring polynomial of an abstract
$G_n$-colored graph.
\end{prop}

\begin{proof} Assume that $d(g^*)=0$. We directly construct the required abstract
$G_n$-colored graph as follows:

\vskip .2cm
First, write $g^*=\sum_{i=1}^\ell s_{i,1}\cdots s_{i, n}$ such that each $s_{i,1}\cdots s_{i, n}$
is the dual monomial of $t_{i,1}\cdots t_{i, n}$, and then take $\ell$ points $v_1, ..., v_\ell$ as
vertices,  which are labeled by monomials $s_{1,1}\cdots s_{1, n}, ...$, $s_{\ell,1}\cdots s_{\ell, n}$,
respectively.

\vskip .2cm
 Next, for each vertex $v_i$, the monomial $s_{i,1}\cdots s_{i, n}$ corresponding to $v_i$
 contains $n$ monomials
$s_{i,1}\cdots \widehat{s}_{i,j}\cdots s_{i, n}$ of degree $n-1$, $j=1, ..., n$.
These $n$ monomials of degree $n-1$ are distinct since $g$ is faithful.
Then we can use $n$ segments to make a bouquet with $v_i$ as a common endpoint by gluing only an
endpoint of each segment to $v_i$, and further we use the
$n$ monomials $s_{i,1}\cdots \widehat{s}_{i,j}\cdots s_{i, n}$ of degree
$n-1$ $(j=1, ..., n)$ to label these $n$ segments. So we can exactly get $\ell$ such bouquets
with $v_1, ..., v_\ell$ as their common endpoints respectively.
Since $d(g^*)=\sum_{i=1}^\ell\sum_{j=1}^n s_{i,1}\cdots \widehat{s}_{i,j}\cdots s_{i, n}=0$,
this means that all $s_{i,1}\cdots \widehat{s}_{i,j}\cdots s_{i, n}\ (i=1, ..., \ell, j=1, ..., n)$
exactly appear in pairs. Moreover, we may use those $\ell$ labeled bouquets to produce an $n$-valent
regular graph $\Gamma$ with $\{v_1, ..., v_\ell\}$ as its vertex set by pairwise gluing segments with
same labels
 together along non-common endpoints of bouquets.

\vskip .2cm Our next procedure is to do a change of labels on all
edges of $\Gamma$. We see easily from the pairing (\ref{pairing})
that each factor $t_{i, j}$ in $t_{i,1}\cdots t_{i, n}$ uniquely
corresponds to a monomial of degree $n-1$ in $s_{i,1}\cdots s_{i,
n}$. Without loss of generality, we may assume that $t_{i, j}$
exactly corresponds to $s_{i,1}\cdots \widehat{s}_{i,j}\cdots s_{i,
n}$. Then we can  give new labels on all edges of  $\Gamma$ by using
$t_{i,j}$ to replace $s_{i,1}\cdots \widehat{s}_{i,j}\cdots s_{i,
n}$ $(i=1, ..., \ell, j=1,..., n)$. Moreover, we conclude by
Lemma~\ref{local} that the graph $\Gamma$ with new labels on edges
is exactly the required abstract $G_n$-colored graph.
\end{proof}

Together with Theorem~\ref{color poly} and Propositions~\ref{p1}--\ref{p2}, we complete the proof of Theorem~\ref{main result}.

\section{The conjecture $(\star)$ and the basic structure of $\mathfrak{M}_*$} \label{app-1}

Now let $\mathcal{V}_n$ be the linear space over $\Z_2$ produced by
all faithful $G_n$-polynomials $g$ in $\Z_2[\widehat{\Hom(G_n,
\Z_2)}]$ with $d(g^*)=0$. By $\mathcal{V}_n^*$ we denote the linear
space over $\Z_2$ formed by the dual polynomials of those
polynomials in $\mathcal{V}_n$. Then $\mathcal{V}_n^*$ is clearly
isomorphic to $\mathcal{V}_n$. Furthermore£¬ we have from
Theorems~\ref{main result} and~\ref{color poly} that
\begin{prop}\label{iso}
$\mathcal{Z}_n(G_n)$ is isomorphic to $\mathcal{V}_n^*$.
\end{prop}

 Since we know from \cite{l1} that
$\mathcal{Z}_n(G_n)=(0)$ if $n=1$, throughout the following assume
that $n>1$.

\subsection{$G_n$-colorings on the products of simplices}
We first do an analysis for the product of simplices with a
$G_n$-coloring, which can provide us much insight into the  study of
$\mathcal{V}_n^*$.

\vskip .2cm
 Let $P^n$ be the product $\Delta^{n_1}\times\cdots\times
\Delta^{n_r}$ of some simplices $\Delta^{n_i}$, $i=1, ..., r$, such
that $P^n$ admits a $G_n$-coloring $\lambda$, where
$n_1+\cdots+n_r=n$. Then $P^n$ has $n+r$ facets,  which consist of
$\Delta^{n_1}\times\cdots\times\Delta^{n_{i-1}}\times
F_j^{(i)}\times\Delta^{n_{i+1}}\times\cdots\times\Delta^{n_r}$
(denoted by $F_{i, j}$), $j=1, ..., n_i+1$, $i=1, ..., r$, where
$F_1^{(i)}, ..., F_{n_i}^{(i)},
F_{n_i+1}^{(i)}$ denote the $n_i+1$ facets of $\Delta^{n_i}$. 
Let $\lambda_i$ be the restriction to $\Delta^{n_i}$ of $\lambda$,
so $\lambda_i(F_j^{(i)})=\lambda(F_{i,j})$ (see
Subsection~\ref{product-coloring}). Then by Lemma~\ref{pro}, each
$\lambda_i$ is a nice $G_n$-coloring, and by
Proposition~\ref{formula} we have that
$$g_{(P^n, \lambda)}=\prod_{i=1}^rg_{(\Delta^{n_i}, \lambda_i)}.$$
Let $\{\rho_1^*, ..., \rho_n^*\}$ denote the standard basis of
$\Hom(\Z_2, G_n)$, where each $\rho_i^*$ is defined by $a\longmapsto
(\underbrace{0,...,0}_{i-1}, a, 0,...,0)$ (see Subsection~\ref{dual
poly}). Consider the following $n$ facets
\begin{equation}\label{facet}
\underbrace{F_{1,1}, ..., F_{1, n_1}}_{n_1}, ..., \underbrace{F_{r, 1}, ..., F_{r, n_r}}_{n_r}
\end{equation}
 which meet at a vertex $v_0$ of $P^n$.
Without loss of generality, assume that  these $n$ facets above
meeting at $v_0$ are colored by $\rho_1^*, ..., \rho_n^*$
respectively (if not, we can always use an automorphism of
$\Hom(\Z_2, G_n)$ to obtain the desired coloring at $v_0$), so
$$\underbrace{\lambda_1(F_1^{(1)})=\rho_1^*, ...,
\lambda_1(F_{n_1}^{(1)})=\rho^*_{n_1}}_{n_1}, ...,
\underbrace{\lambda_r(F_1^{(r)})=\rho^*_{n_1+\cdots+n_{r-1}+1}, ...,
\lambda_r(F_{n_r}^{(r)})=\rho^*_n}_{n_r}.$$
    We note that if a facet $F_{i, j}$ ($1\leq j\leq n_i$) in the $n$ facets of (\ref{facet}) is
    replaced by $F_{i, n_i+1}$, then the resulting $n$ facets  still meet at a vertex of $P^n$.
Since the $G_n$-coloring of $n$ facets meeting at each vertex of
$P^n$ gives a basis of $\Hom(\Z_2, G_n)$,   we see easily that
$$
\begin{cases}\lambda_1(F_{n_1+1}^{(1)})&=\rho_1^*+\cdots+\rho^*_{n_1}+\beta_1\\
& \cdots\\
\lambda_i(F_{n_i+1}^{(i)})&=\rho^*_{\sum_{j=1}^{i-1}n_j+1}+\cdots+
\rho^*_{\sum_{j=1}^{i}n_j}+\beta_i\\
&\cdots\\
\lambda_r(F_{n_r+1}^{(r)})&=\rho_{n_1+\cdots+n_{r-1}+1}^*+\cdots+\rho^*_{n}+\beta_r\end{cases}$$
such that each $\beta_i\in \Hom(\Z_2, G_n)$ satisfies the property
$(*)$ that $\beta_i$ is a linear combination of $\{\rho_1^*, ...,
\rho^*_n\}\setminus\{\rho^*_{n_1+\cdots+n_{i-1}+1}, ...,
\rho^*_{n_1+\cdots+ n_i}\}.$  Thus, for each $i$, we have that
$$g_{(\Delta^{n_i}, \lambda_i)}=d\Big(\rho^*_{\sum_{j=1}^{i-1}n_j+1}\cdots
\rho^*_{\sum_{j=1}^{i}n_j}
\big(\rho^*_{\sum_{j=1}^{i-1}n_j+1}+\cdots+
\rho^*_{\sum_{j=1}^{i}n_j}+\beta_i\big)\Big).$$ This gives
\begin{lem}\label{simplex}
With the above assumption, the $G_n$-coloring polynomial of $(P^n,
\lambda)$ is $$g_{(P^n,\lambda)}
=\prod_{i=1}^rd\Big(\rho^*_{\sum_{j=1}^{i-1}n_j+1}\cdots
\rho^*_{\sum_{j=1}^{i}n_j}
\big(\rho^*_{\sum_{j=1}^{i-1}n_j+1}+\cdots+
\rho^*_{\sum_{j=1}^{i}n_j}+\beta_i\big)\Big).$$
\end{lem}
Note that if some $n_i=1$ and $g_{(P^n, \lambda)}\not=0$, then
$g_{(\Delta^{n_i}, \lambda_i)}\not=0$, so $\beta_i$ must be nonzero.
This means that two facets of $\Delta^{n_i}=\Delta^1$ are colored by
two different elements.

\begin{lem}
With the above notion, there is at least one $\beta_i$ in the
expression of $g_{(P^n, \lambda)}$ such that $\beta_i=0$.
\end{lem}
\begin{proof}
When $r=1$ (i.e., $P^n=\Delta^n$), clearly we must have that
$$g_{(P^n,
\lambda)}=d\big(\rho_1^*\cdots\rho_n^*(\rho_1^*+\cdots+\rho_n^*)\big),$$
so $\beta_1=0$ as desired. When $r>1$, without loss of generality,
suppose that there is some positive integer $k$ with $k\leq r$ such that
for all $1\leq i\leq k$, $\beta_i\not=0$, so $\beta_1$ is nonzero. Furthermore, assume that
the expression of $\beta_1$ contains the term $\rho^*_{n_1+1}$.
\vskip .2cm

\noindent {\bf Claim A.} {\em For each $i\leq k$, $\beta_i$ is a
linear combination of  $\rho^*_{\sum_{j=1}^in_j+1}, ..., \rho^*_n$.}

\vskip.1cm
 If $2\leq k$, then $\beta_2\not=0$.
We have known that $$\beta_2\in\H(\Z_2, G_n)\setminus
\text{Span}\{\rho^*_{n_1+1}, ..., \rho^*_{n_1+ n_2}\}.$$ If the
expression of $\beta_2$ contains at least one of $\rho_1^*, ...,
\rho_{n_1}^*$, without loss of generality, we  assume that $\beta_2$
contains the term $\rho^*_1$ in its expression. Consider the
following $n$ facets
$$\underbrace{F_{1, 2}, ..., F_{1, n_1+1}}_{n_1}, \underbrace{F_{2,2}, ..., F_{2, n_2+1}}_{n_2}, \underbrace{F_{3,1},
..., F_{3, n_3}}_{n_3}, ..., \underbrace{F_{r,1}, ..., F_{r,
n_r}}_{n_r}$$ which can meet at a vertex of $P^n$. These $n$ facets
determine  the following $n$ elements
\begin{align*}
&\lambda_1(F_2^{(1)})=\rho_2^*, ...,
\lambda_1(F^{(1)}_{n_1})=\rho^*_{n_1}, \lambda_1(F^{(1)}_{n_1+1})
=\rho_1^*+\cdots+\rho_{n_1}^*+\beta_1\\
& \lambda_2(F_2^{(2)}) =\rho^*_{n_1+2}, ...,
\lambda_2(F_{n_2}^{(2)})=\rho^*_{n_1+n_2},
\lambda_2(F^{(2)}_{n_2+1})=\rho^*_{n_1+1}+\cdots+\rho^*_{n_1+n_2}+\beta_2\\
&\lambda_3(F_1^{(3)})=\rho^*_{n_1+n_2+1}, ...,
\lambda_3(F^{(3)}_{n_3})=\rho^*_{n_1+n_2+n_3}\\
& ...\\
& \lambda_r(F_1^{(r)})=\rho^*_{n_1+\cdots+n_{r-1}+1}, ...,
\lambda_3(F^{(r)}_{n_r})=\rho^*_n
\end{align*}
which should form a basis of $\H(\Z_2, G_n)$. However, these $n$
elements above are actually linearly dependent. This is because
cosets
 $$\lambda_1(F^{(1)}_{n_1+1})+A=\lambda_2(F^{(2)}_{n_2+1})+A=\rho_1^*+\rho^*_{n_1+1}+A$$
in the quotient group $\H(\Z_2, G_n)/A$ where $A$ is the subgroup
generated by other $n-2$
  elements above except for $\lambda_1(F^{(1)}_{n_1+1})$ and $\lambda_2(F^{(2)}_{n_2+1})$. Thus, this case is impossible, so
$\beta_2$ is a linear combination of $\rho^*_{n_1+n_2+1}, ...,
\rho^*_n$. This also implies that $r$ must be greater than 2.
Without loss of generality, we may further assume that the
expression of $\beta_2$ contains the term $\rho_{n_1+n_2+1}^*$.

\vskip .2cm Now  suppose inductively that for each $i< k$, $\beta_i$
is a linear combination of
 $\rho^*_{\sum_{j=1}^in_j+1}, ..., \rho^*_n$, and the expression of $\beta_i$
 contains the term $\rho^*_{\sum_{j=1}^in_j+1}$. Consider the case in which $i=k$.
 First, we note that up to combinatorial equivalence, $P^n$ is independent of the choice
 of product orderings of $\Delta^{n_1}, ..., \Delta^{n_r}$. In other words,
 for any permutation $\{i_1, ..., i_r\}$ of $\{1, ..., r\}$, $\Delta^{n_{i_1}}\times\cdots\times\Delta^{n_{i_r}}$ is combinatorially equivalent to
 $\Delta^{n_1}\times\cdots\times\Delta^{n_r}$. Thus, by Remark~\ref{pr} we can identify
  $(\Delta^{n_1}, \lambda_1)\times\cdots\times (\Delta^{n_r}, \lambda_r)$ with $(\Delta^{n_{i_1}}, \lambda_{i_1})\times\cdots\times (\Delta^{n_{i_r}}, \lambda_{i_r})$. With this understood,  we  change the ordering of $\{1, ..., k, k+1, ...,  r\}$ into a new ordering
 $\{2, ..., k, 1, k+1, ..., r\}$, so that $\beta_2, ..., \beta_k$ in this new ordering can be regarded as
 $\beta_1, ..., \beta_{k-1}$ in the ordering of $\{1, .., r\}$. Then
 we can use the inductive hypothesis to obtain that $\beta_{k}$ does not contain
 anyone of $\rho^*_{n_1+1}, ..., \rho^*_{n_1+\cdots+n_{k}}$ in its
 expression. Similarly, consider the following permutations of $\{1,..., r\}$
 $$\{3, ..., k, 1, 2, k+1, ..., r\},  \{4, ..., k, 1, 2, 3, k+1, ..., r\}, ..., \{k, 1, ..., k-1, k+1, ..., r\}$$
 we conclude that for $i<k$, the expression of $\beta_i$ does not contain anyone of $$\{\rho_1^*, ..., \rho^*_{n_1+\cdots+n_k}\}\setminus
 \{\rho^*_{n_1+\cdots+n_{i}+1}, ..., \rho^*_{n_1+\cdots+n_{i+1}}\}.$$

 It remains to show that the expression of $\beta_{k}$ does not contain anyone of $\rho^*_1, ..., \rho^*_{n_1}$.
If not, assume that the expression of $\beta_{k}$ contains the term
$\rho^*_1$. For $i\leq k-1$, write
 \begin{align*}
\lambda_i(F_{n_i+1}^{(i)})=&\sum_{l=1}^{n_i}\rho_{n_1+\cdots+n_{i-1}+l}^*+\rho^*_{n_1+\cdots+n_i+1}
+\gamma_i+\gamma'_i
\end{align*}
and for $i=k$, write
\begin{align*}
\lambda_{k}(F_{n_{k}+1}^{(k)})=&\rho_{n_1+\cdots+n_{k-1}+1}^*+\cdots+\rho^*_{n_1+\cdots+n_{k}}+
\rho^*_1+\gamma_{k}+\gamma'_{k}
\end{align*}
where  for $i< k$, $\gamma_i$ is a linear combination of $\rho^*_{n_1+\cdots+n_{i}+2}, ...,
\rho^*_{n_1+\cdots+n_{i+1}}$ and $\gamma_{k}$ is a linear combination of $\rho^*_2, ...,
\rho^*_{n_1}$, and for $i\leq k$, $\gamma'_i $ is a
linear combination of $\rho^*_{\sum_{j=1}^kn_j+1}, ..., \rho^*_n$. Consider the following $n$ facets
$$\underbrace{F_{1, 2}, ..., F_{1, n_1+1}}_{n_1}, ..., \underbrace{F_{k,2}, ..., F_{k, n_k+1}}_{n_k}, \underbrace{F_{k+1,1},
..., F_{k+1, n_{k+1}}}_{n_{k+1}}, ..., \underbrace{F_{r,1}, ...,
F_{r, n_r}}_{n_r}$$ which can meet at a vertex of $P^n$. Then these
$n$ facets determine the following $n$ linearly independent elements
of $\H(\Z_2, G_n)$ via $\lambda$
\begin{align*}
&\underbrace{\rho_2^*, ..., \rho^*_{n_1}, \lambda_1(F^{(1)}_{n_1+1})}_{n_1}, ...,  
\underbrace{\rho^*_{\sum_{j=1}^{k-1}n_j+2}, ...,\rho^*_{\sum_{j=1}^kn_j},
\lambda_k(F^{(k)}_{n_k+1})}_{n_k},\\
&\underbrace{\rho^*_{\sum_{j=1}^kn_j+1}, ..., \rho^*_{\sum_{j=1}^{k+1}n_j}}_{n_{k+1}},
 ...,
\underbrace{\rho^*_{\sum_{j=1}^{r-1}n_j+1}, ..., \rho^*_n}_{n_r}.
\end{align*}
However, by direct calculations we have that the coset
$$\lambda_1(F^{(1)}_{n_1+1})+\cdots+\lambda_k(F^{(k)}_{n_k+1})+B=0 \text{ in }  \H(\Z_2, G_n)/B$$
which is impossible, where $B$ is the subgroup generated by
$\rho_2^*, ..., \rho^*_{n_1}, ...,$
$\rho^*_{\sum_{j=1}^{k-1}n_j+2}$, $..., \rho^*_{\sum_{j=1}^kn_j}$,
$\rho^*_{\sum_{j=1}^kn_j+1}, ..., \rho^*_n$. This completes the
induction and the proof of Claim A.

\vskip .2cm By Claim A, we see that it is impossible to have $\beta_i\not=0$ for all $i\leq r$
 since $r\leq n$.  Thus, there must be at least one
$\beta_i=0$.
\end{proof}

\begin{cor}
If $P^n$ is an $n$-cube  with a $G_n$-coloring $\lambda$, then
$g_{(P^n, \lambda)}=0$. Moreover, each real Bott manifold bounds
equivariantly.
\end{cor}

\subsection{Structure of $\mathcal{V}_n^*$}
Let $h$ be a polynomial in $\mathcal{V}_n^*$. Then $d(h)=0$,  and by
Lemma~\ref{commut}, we have that $d(\sigma(h))=\sigma(d(h))=0$ for
any automorphism $\sigma\in\text{\rm Aut}(\Hom(\Z_2, G_n))$, so
$\sigma(h)\in \mathcal{V}_n^*$. By Corollary~\ref{diff}, we have
that there is  a squarefree homogeneous polynomial $\varphi$ of
degree $n+1$ in $\Z_2[\widehat{\Hom(\Z_2, G_n)}]$ such that
$d(\varphi)=h$, and all $n+1$ degree-one factors in each monomial of
$\varphi$ contain a basis of $\Hom(\Z_2, G_n)$.

\vskip .2cm Now let $s_1\cdots s_{n+1}$ be a monomial of $\varphi$.
Since $s_1, ..., s_{n+1}$ contain a basis of $\Hom(\Z_2, G_n)$, we
can apply an automorphism of $\Hom(\Z_2, G_n)$ to change this basis
into the standard basis $\{\rho_1^*, ..., \rho_n^*\}$, so without
loss of generality we may assume that $s_1, ..., s_{n+1}$ exactly
contain the standard basis $\{\rho_1^*, ...,
\rho_n^*\}$ of $\Hom(\Z_2, G_n)$. 
Furthermore,  all possible choices of
$s_1\cdots s_{n+1}$ are the following
\begin{equation*}
\begin{cases}
\rho_1^*\cdots\rho_n^*(\rho_1^*+\cdots+\rho_n^*) \\
\rho_1^*\cdots\rho_n^*(\rho_1^*+\cdots+\rho_{n-1}^*)\\
\ \ \ \ \ \ \ \ \ \ \ \vdots\\
\rho_1^*\cdots\rho_n^*(\rho_1^*+\rho_2^*).
\end{cases}
\end{equation*}
Obviously,
$d\big(\rho_1^*\cdots\rho_n^*(\rho_1^*+\cdots+\rho_n^*)\big)$
belongs to $\mathcal{V}^*_n$. If $\varphi$ contains the monomial
$\rho_1^*\cdots\rho_n^*(\rho_1^*+\cdots+\rho_{n-1}^*)$, then it must
contain the monomial
 $\rho_1^*\cdots\rho_{n-1}^*(\rho_1^*+\cdots+\rho_{n-1}^*)(\rho_n^*+s)$
 where $s\not=0$ is a linear combination of
 $\rho_1^*, ..., \rho_{n-1}^*$. Then we see easily that
 $$d\big(\rho_1^*\cdots\rho_n^*(\rho_1^*+\cdots+\rho_{n-1}^*)+\rho_1^*\cdots\rho_{n-1}^*(\rho_1^*+\cdots+\rho_{n-1}^*)
 (\rho_n^*+s)\big)\in \mathcal{V}^*_n.$$
 Generally, if $\varphi$ contains the monomial
 $\rho_1^*\cdots\rho_n^*(\rho_1^*+\cdots+\rho_{i}^*)$ with $i>1$, then
 we can write
 $$\varphi=\rho_1^*\cdots\rho_i^*(\rho_1^*+\cdots+\rho_{i}^*)f+\varphi'$$
 where $f$ is a squarefree homogeneous polynomial of degree $n-i$ such
 that the cosets of
  all degree-one factors modulo $\text{\rm Span}\{\rho_1^*, ..., \rho_i^*\}$ of each monomial of $f$ are
  also linearly independent in
  $\Hom(\Z_2, G_n)/\text{\rm Span}\{\rho_1^*, ..., \rho_i^*\}$,   and each monomial of $\varphi'$ does not contain
 $\rho_1^*\cdots\rho_i^*(\rho_1^*+\cdots+\rho_{i}^*)$. Since
 $$h=d(\varphi)=\rho_1^*\cdots\rho_i^*(\rho_1^*+\cdots+\rho_{i}^*)d(f)+
 fd\big(\rho_1^*\cdots\rho_i^*(\rho_1^*+\cdots+\rho_{i}^*)\big)+d(\varphi')$$
 belongs to $\mathcal{V}_n^*$, this forces $d(f)$ to be zero. Thus, we conclude that
 $$d\big(\rho_1^*\cdots\rho_i^*(\rho_1^*+\cdots+\rho_{i}^*)f\big)\in \mathcal{V}^*_n.$$

\vskip .1cm

Combining the above argument gives the following result.

\begin{lem}\label{basis}
  $\mathcal{V}^*_n$ is generated by the  polynomials of the following form
$$\sigma(d(\varphi_i)), \ \ 1<i\leq n$$
where $\sigma\in \text{\rm Aut}(\Hom(\Z_2, G_n))$, and
$\varphi_i=\rho_1^*\cdots\rho_i^*(\rho_1^*+\cdots+\rho_{i}^*)f_i$
with $f_i$ having the property that
        $f_i$ is a nonzero squarefree homogeneous polynomial of degree $n-i$ such that  $d(f_i)=0$,
         and if $i<n$, the cosets of $\text{\rm Span}\{\rho_1^*, ..., \rho_i^*\}$ in $\Hom(\Z_2, G_n)$
         determined by all degree-one factors of
         each monomial of $f_i$ are also linearly independent in
  $\Hom(\Z_2, G_n)/\text{\rm Span}\{\rho_1^*, ..., \rho_i^*\}$, and if $i=n$, $f_n=1$.
\end{lem}

\begin{defn}\label{colorable}
Let $f$ be a squarefree homogeneous polynomial in
$\Z_2[\widehat{\Hom(\Z_2, G_n)}]$ with $0<\deg f\leq n$. We say that
$f$ is {\em $G_n$-colorable} if $f$ is  the sum of the coloring
polynomials of  some $\deg f$-dimensional simple convex  polytopes
$P_1,  ..., P_l$ with nice $G_n$-colorings $\lambda_1, ...,
\lambda_l$ respectively, where each $P_i$ is a product of simplices.
\end{defn}

It is easy to see the following property:
\begin{enumerate}
 \item[$(**)$] If $f$ is $G_n$-colorable,  then for any automorphism
 $\sigma\in \text{Aut}(\Hom(\Z_2,G_n))$, $\sigma(f)$ is $G_n$-colorable .
     \end{enumerate}

Let $\mathcal{S}$ denote the set formed by those nonzero squarefree
homogeneous polynomials $f\in \Z_2[\widehat{\Hom(\Z_2, G_n)}]$ with
the following properties:
\begin{enumerate}
\item[(1)]
   $d(f)=0$;
   \item[(2)]
   $0<\deg f\leq n$;
 \item[(3)] each $f$ is associated with a  subspace $V_f$ of dimension $n-\deg f$ in $\Hom(\Z_2, G_n)$
 such that
 the cosets of $V_f$ in $\Hom(\Z_2, G_n)$ determined by all degree-one factors of each monomial of $f$ are  linearly independent
 in $\Hom(\Z_2, G_n)/V_f$.
\end{enumerate}

\begin{prop} \label{key}
Let  $f$ be a polynomial in  $\mathcal{S}$. Then $f$ is
$G_n$-colorable.
\end{prop}
\begin{proof}
We shall perform an induction on $\deg f$. If $\deg f=1$, since
$d(f)=0$, then there must be an even number of distinct monomials
$s_1, ..., s_{2r}$ of degree one in $\Z_2[\widehat{\Hom(\Z_2,
G_n)}]$ such that $f=s_1+\cdots+s_{2r}$. Each pair $(s_{2i-1},
s_{2i})$ can give a $G_n$-coloring $\lambda_i$ on $\Delta^1$ with
two facets colored by $s_{2i-1}, s_{2i}$ respectively,  so that the
coloring polynomial $g_{(\Delta^1, \lambda_i)}$ of $(\Delta^1,
\lambda_i)$ is $s_{2i-1}+s_{2i}$, where $1\leq i\leq r$. Thus
$f=g_{(\Delta^1, \lambda_1)}+\cdots+g_{(\Delta^1, \lambda_r)}$, and
so $f$ is $G_n$-colorable.

\vskip .1cm Now suppose inductively that $f$ is $G_n$-colorable if
$\deg f\leq k<n$. Consider the case in which $\deg f=k+1$. Without
loss of generality, assume that $V_f$ is generated by $\rho_1^*,
..., \rho^*_{n-k-1}$, and there is at least one monomial of $f$
which contains the degree-one factor of the form $s$ with $s+V_f=
\rho^*_{n-k}+V_f$ in $\Hom(\Z_2, G_n)/V_f$ (if not, we may always
use an automorphism of $\H(\Z_2, G_n)$ to do this). Then we may
write $f$ as
\begin{equation} \label{eq1}
f=s_1a_1+\cdots+s_ra_r+f'\end{equation} where $s_1, ..., s_r$ are
distinct degree-one elements in $\Z_2[\widehat{\Hom(\Z_2, G_n)}]$
such that $s_i+V_f=\rho^*_{n-k}+V_f, i=1, ...,r$, and after reducing
modulo $V_f$, each monomial of $f'$  contains no degree-one factor
$\rho^*_{n-k}$. Since $d(f)=0$, by a direct calculation we have that
$$\big(s_1d(a_1)+\cdots+s_rd(a_r)\big)+\big(a_1+\cdots+a_r+d(f')\big)=0.$$
Since $a_1+\cdots+a_r+d(f')$ contains no degree-one factors $s_1,
..., s_r$, we can conclude that
$s_1d(a_1)+\cdots+s_rd(a_r)=a_1+\cdots+a_r+d(f')=0$ in
$\Z_2[\widehat{\Hom(\Z_2, G_n)}]$. Moreover, from
$s_1d(a_1)+\cdots+s_rd(a_r)=0$, an easy argument shows that
$d(a_i)=0, i=1, ..., r$, so  $a_i\in \mathcal{S}, i=1, ..., r$. From
$a_1+\cdots+a_r+d(f')=0$, we have that
$a_r=a_1+\cdots+a_{r-1}+d(f')$, so (\ref{eq1}) becomes
\begin{eqnarray*}
f &=& s_1a_1+\cdots +s_{r-1}a_{r-1}+s_r\big(a_1+\cdots+a_{r-1}+d(f')\big)+f'\\
  &= &(s_1+s_r)a_1+\cdots+ (s_{r-1}+s_r)a_{r-1}+s_rd(f')+f'\\
  &=& (s_1+s_r)a_1+\cdots+ (s_{r-1}+s_r)a_{r-1}+ d(s_rf').
\end{eqnarray*}
Since $\deg a_i=k<k+1$ and $a_i\in \mathcal{S}$, by induction
hypothesis, we have that each $a_i$ is $G_n$-colorable. Obviously,
$s_1+s_r, ..., s_{r-1}+s_r$ are all $G_n$-colorable. Thus, by
Proposition~\ref{formula}, it follows that $(s_1+s_r)a_1+\cdots+
(s_{r-1}+s_r)a_{r-1}$ is $G_n$-colorable.

\vskip .1cm Now, to complete the proof, it merely needs to show that
$d(s_rf')$ is $G_n$-colorable. Our argument proceeds as follows:
Write $\psi=s_rf'$ and $s_r=\rho^*_{n-k}+\delta_{n-k}$ where
$\delta_{n-k}\in V_f$. As is known as above, after reducing modulo
$V_f$,  each monomial of $f'$ contains no degree-one factor
$\rho^*_{n-k}$. Thus,
 $\psi$ is actually a squarefree homogeneous polynomial of degree $k+2$ in
 $\Z_2[\widehat{\Hom(\Z_2, G_n)}]$.
Now for each monomial $s_r\alpha_1\cdots\alpha_{k+1}$ of $\psi$, we
know by the definition of $\mathcal{S}$ that $\rho^*_1, ...,
\rho^*_{n-k-1}, \alpha_1, ..., \alpha_{k+1}$ form a basis of
$\H(\Z_2, G_n)$. Then we may choose an automorphism $\sigma$ of
$\Hom(\Z_2, G_n)$ such that the restriction of $\sigma$ to $V_f$ is
the identity, and $\sigma(s_r\alpha_1\cdots\alpha_{k+1})$ is of the
following form
\begin{eqnarray*}
\rho^*_{n-k}\cdots \rho^*_n
(\rho^*_{n-k}+\cdots+\rho^*_{n-k+j}+\beta)
\end{eqnarray*}
where $1\leq j\leq k$, and $\beta\in V_f$. Write
$$\sigma(\psi)=\rho^*_{n-k}\cdots \rho^*_{n-k+j}
(\rho^*_{n-k}+\rho^*_{n-k+1}+\cdots+\rho^*_{n-k+j}+\beta)b_j+\psi'$$
where each monomial of $\psi'$ does not contain $\rho^*_{n-k}\cdots
\rho^*_{n-k+j}
(\rho^*_{n-k}+\rho^*_{n-k+1}+\cdots+\rho^*_{n-k+j}+\beta)$, and the
cosets of all degree-one factors modulo $\text{Span}\{\rho_1^*, ...,
\rho^*_{n-k+j}\}$ of each monomial of $b_j$ are linearly independent
in $\H(\Z_2, G_n)/\text{Span}\{\rho_1^*, ..., \rho^*_{n-k+j}\}$ if
$j<k$ (otherwise, $b_j=1$ if $j=k$). In a similar way to the
argument of Lemma~\ref{basis}, by calculation we have that
\begin{align*} d(\sigma(\psi))= & \rho^*_{n-k}\cdots \rho^*_{n-k+j}
(\rho^*_{n-k}+\rho^*_{n-k+1}+\cdots+\rho^*_{n-k+j}+\beta)d(b_j)\\
&+b_jd\big(\rho^*_{n-k}\cdots \rho^*_{n-k+j}
(\rho^*_{n-k}+\rho^*_{n-k+1}+\cdots+\rho^*_{n-k+j}+\beta)\big)+d(\psi').
\end{align*}
Since $ d(\sigma(\psi))\in \mathcal{S}$, this forces $d(b_j)$ to be
zero, so $b_j\in \mathcal{S}$ if $j<k$. This means that $\psi$ is a
linear combination of the polynomials $\tau(\psi_j)$, where
$\psi_j=\rho^*_{n-k}\cdots
\rho^*_{n-k+j}(\rho^*_{n-k}+\rho^*_{n-k+1}+\cdots+\rho^*_{n-k+j}+\beta)b_j$
with $b_j\in \mathcal{S}$ if $j<k$ and $b_k=1$ if $j=k$, and $\tau$
is chosen in $\text{Aut}(\Hom(\Z_2, G_n))$ such that the restriction
of $\tau$ to $V_f$ is the identity. We see from the proof of
Lemma~\ref{simplex} that $d\big(\rho^*_{n-k}\cdots
\rho^*_{n-k+j}(\rho^*_{n-k}+\rho^*_{n-k+1}+\cdots+\rho^*_{n-k+j}+\beta)\big)$
is $G_n$-colorable. When $j<k$, we have that $\deg b_j=k-j<\deg
f=k+1$, so by induction hypothesis again, $b_j$ is $G_n$-colorable.
 Then  we obtain by Proposition~\ref{formula} that for $1\leq j\leq k$,
 $$d(\psi_j)=b_jd\big(\rho^*_{n-k}\cdots \rho^*_{n-k+j}(\rho^*_{n-k}+\rho^*_{n-k+1}+\cdots+
 \rho^*_{n-k+j}+\beta)\big)$$ is $G_n$-colorable.
Moreover, since $d(\psi)$ is a linear combination of the polynomials
$d(\tau(\psi_j))$ and since $ d(\tau(\psi_j))=\tau(d(\psi_j))$ by
Lemma~\ref{commut},  we then conclude by the property $(**)$ that
$d(\psi)$ is $G_n$-colorable. This completes the proof.
\end{proof}

\begin{cor}\label{colorable1}
Each polynomial of $\mathcal{V}_n^*$ is $G_n$-colorable.
\end{cor}

\begin{proof}
This is because each polynomial of $\mathcal{V}_n^*$ also belongs to $\mathcal{S}$.
\end{proof}

\begin{cor}\label{generator}
$\mathcal{Z}_n(G_n)$ is generated by the classes of small covers
over $\Delta^{n_1}\times\cdots\times \Delta^{n_\ell}$ where
$n_1+\cdots +n_\ell=n$.
\end{cor}

\begin{proof}
We note  that by \cite[\S1.5. The basic construction]{dj},  each
$G_n$-colored simple convex $n$-polytope
 $(P^n, \lambda)$ can reconstruct an $n$-dimensional small cover $M(P^n,\lambda)$ over $P^n$,
 and  by Proposition~\ref{small},   $\phi_n(\{M(P^n, \lambda)\})$ is the dual polynomial of
 $g_{(P^n, \lambda)}\in\mathcal{V}_n^*$. Then  Corollary~\ref{generator}  immediately follows from Proposition~\ref{iso} and Corollary~\ref{colorable1}.
\end{proof}

\subsection{Proofs of Theorems~\ref{conj}--\ref{ring}} \label{proof of thm}

Now we can first give the proof of Theorem~\ref{ring}.\vskip .2cm

\noindent {\em Proof of Theorem~\ref{ring}.} This is a direct
consequence of Corollary~\ref{generator} and
Proposition~\ref{formula}. \hfill $\Box$

\vskip .2cm

Next let us finish the proof of Theorem~\ref{conj}.

\vskip .2cm

\noindent {\em Proof of Theorem~\ref{conj}.}
    We have known from the proof of Corollary~\ref{generator} that the $G_n$-coloring polynomial
    of each $G_n$-colored simple convex $n$-polytope $(P^n, \lambda)$ uniquely determines an equivariant unoriented cobordism
   class containing the small cover $M(P^n, \lambda)$ over $P^n$ as its representative.
   Thus, by Proposition~\ref{iso} we only need to show that each polynomial in $\mathcal{V}_n^*$ is
   the $G_n$-coloring polynomial of a $G_n$-colored simple convex $n$-polytope. In order to prove this, by Corollary~\ref{colorable1} it suffices  to show that

 \vskip .2cm
 \noindent
 {\bf Claim B.} {\em Let $f_1$ and $f_2$ be two polynomials in $\mathcal{V}^*_n$.
 If $f_1$ and $f_2$ are the coloring polynomials of two $G_n$-colored simple convex $n$-polytopes
 $(P_1, \lambda_1)$ and $(P_2, \lambda_2)$ respectively, then $f_1+f_2$ is the coloring polynomial
 of a $G_n$-colored simple convex $n$-polytope, too.}

\vskip .2cm

If $f_1$ and $f_2$ have the same monomial $s_1\cdots s_n$, then
there must be a vertex $v_1$ of $P_1$ and a vertex $v_2$ of $P_2$
such that the coloring monomials of $v_1$ and $v_2$ are same and
equal to $s_1\cdots s_n$, i.e.,
$\lambda_{1v_1}=\lambda_{2v_2}=s_1\cdots s_n$. Now by
Proposition~\ref{sum formula} we can perform a connected sum
$P_1\sharp_{v_1, v_2}P_2$ of $(P_1, \lambda_1)$ and $(P_2,
\lambda_2)$ at $v_1$ and $v_2$, such that $P_1\sharp_{v_1, v_2}P_2$
is a simple convex polytope and naturally admits a $G_n$-coloring
$\lambda_1\sharp_{v_1, v_2}\lambda_2$ induced by $\lambda_1$ and
$\lambda_2$, and in particular, $f_1+f_2$ is the coloring polynomial
of $(P_1\sharp_{v_1, v_2}P_2, \lambda_1\sharp_{v_1, v_2}\lambda_2)$.

\vskip .2cm

Suppose that $f_1$ and $f_2$ contain no same monomial.
If $f_1$ contains a monomial $s_1s_2\cdots s_n$ but $f_2$
contains a monomial $\widetilde{s_1}s_2\cdots s_n$, then there must be a vertex $v_1$ of $P_1$ and a vertex $v_2$ of $P_2$ such that $\lambda_{1v_1}=s_1s_2\cdots s_n$ and   $\lambda_{2v_2}=\widetilde{s_1}s_2\cdots s_n$. Consider the colored simple convex polytope $Q_1=\Delta^{n-1}\times \Delta^1$ with coloring $\lambda'$ such that $n$ facets $F_1\times \Delta^1, ..., F_n\times \Delta^1$ are colored by $s_2, ..., s_n, s_2+\cdots +s_n$ and two facets $\Delta^{n-1}\times *_1, \Delta^{n-1}\times *_2$ are colored by $s_1, \widetilde{s_1}$, where
$F_1, ..., F_n$ denote all facets of $\Delta^{n-1}$ and $*_1, *_2$ denote two facets of $\Delta^1$.
Obviously, there are two vertices $u_1, u_2$ of $Q_1$ such that $\lambda'_{u_1}=s_1s_2\cdots s_n$ and
$\lambda'_{u_2}=\widetilde{s_1}s_2\cdots s_n$. Choose a vertex $q_1$ of $Q_1$ such that $\lambda'_{q_1}=s_1s_3\cdots s_n
(s_2+\cdots+s_n)$. Then we can perform a connected sum of $P_1,P_2, Q_1$ to get the following simple convex polytope
$$P=P_1\sharp_{v_1, u_1} Q_1\sharp_{q_1,q_1}Q_1\sharp_{u_2, v_2}P_2$$
which admits a $G_n$-coloring $\lambda$ determined by $\lambda_1,
\lambda_2, \lambda'$. Furthermore, by Proposition~\ref{sum formula}
we see that $f_1+f_2$ is the coloring polynomial of $(P, \lambda)$.

\vskip .2cm
Generally, if $f_1$ contains a monomial $s_1s_2\cdots s_n$ but $f_2$
contains a different monomial $\widetilde{s_1}\widetilde{s_2}\cdots \widetilde{s_n}$, using the above constructed colored polytope $(Q_1, \lambda')$, we first can obtain a colored simple convex polytope $(P_1\sharp_{v_1, u_1} Q_1\sharp_{q_1,q_1}Q_1, \lambda^{(1)})$ such that there is a vertex $p_1$ with its coloring monomial $\lambda_{p_1}^{(1)}=\widetilde{s_1}s_2\cdots s_n$. In the same way as above, we construct a colored simple convex polytope $Q_2=\Delta^{n-1}\times \Delta^1$ with coloring $\lambda''$ such that $n$ facets $F_1\times \Delta^1, ..., F_n\times \Delta^1$ are colored by $\widetilde{s_1}, s_3, ..., s_n, \widetilde{s_1}+s_3+\cdots +s_n$ and two facets $\Delta^{n-1}\times *_1, \Delta^{n-1}\times *_2$ are colored by $s_2, \widetilde{s_2}$. In particular, it is not difficult to see that there are two vertices $p'_1$ and $p''_1$ such that $\lambda''_{p'_1}=\widetilde{s_1}s_2\cdots s_n$
and $\lambda''_{p''_1}=\widetilde{s_1}\widetilde{s_2}s_3\cdots s_n$. Take the vertex $q_2$ in $Q_2$ such that $\lambda''_{q_2}=s_2\cdots s_n(\widetilde{s_1}+s_3+\cdots +s_n)$. Now by doing connected sum we can construct the following colored simple convex polytope
$$(P_1\sharp_{v_1, u_1} Q_1\sharp_{q_1,q_1}Q_1\sharp_{p_1, p'_1}Q_2\sharp_{q_2, q_2}Q_2, \lambda^{(2)})$$
such that there is a vertex $p_2$ with $\lambda^{(2)}_{p_2}=\widetilde{s_1}\widetilde{s_2}s_3\cdots s_n$ in this new colored simple convex polytope.
Continuing this procedure, we can further construct a series of colored simple convex polytopes $Q_3,  ..., Q_n$, so that finally we can obtain a colored simple convex polytope
$$P'=P_1\sharp_{v_1, u_1} Q_1\sharp_{q_1,q_1}Q_1\sharp_{p_1, p'_1}Q_2\sharp_{q_2, q_2}Q_2\sharp_{p_2, p'_2}\cdots
\sharp_{p_{n-1}, p'_{n-1}}Q_n\sharp_{q_n,q_n}Q_n$$ with
$G_n$-coloring $\lambda^{(n)}$ such that there is a vertex $p_n$ of
$P'$ with $\lambda^{(n)}_{p_n}=\widetilde{s_1}\widetilde{s_2}\cdots
\widetilde{s_n}$. Now let $v_2$ be the vertex of $P_2$ such that
$\lambda_{2v_2}=\widetilde{s_1}\widetilde{s_2}\cdots
\widetilde{s_n}$. Then we can get a colored simple convex polytope
$(P'\sharp_{p_n, v_2}P_2, \lambda)$  as desired. Thus, Claim B
holds. This completes the proof of Theorem~\ref{conj}. \hfill$\Box$

\begin{cor} \label{relation}
A faithful $G_n$-polynomial $g\in \Z_2[\widehat{\Hom(G_n,{\Bbb
Z}_2)}]$ belongs to $\IM\phi_n$ if and only if its dual polynomial
 $g^*$ is the $G_n$-coloring polynomial of a $G_n$-colored simple convex polytope $(P^n, \lambda)$.
\end{cor}

\subsection{Determination of $\mathcal{Z}_4(G_4)$}\label{4-dim}

By Proposition~\ref{iso}, it suffices to determine the concrete
structure of $\mathcal{V}_4^*$.

\begin{lem}\label{v}
$\mathcal{V}_4^*$ is generated by the polynomials of the form
$$d\big(s_1s_2(s_1+s_2)\big)d\big(s_3s_4(s_3+s_4+\varepsilon s_1)\big)$$
where $\{s_1, s_2, s_3, s_4\}$ is a basis of $\Hom(\Z_2, G_4)$ and
$\varepsilon=0$ or 1.
\end{lem}

\begin{proof}
 First, Corollary~\ref{colorable1} tells us that  $\mathcal{V}^*_4$ is generated by coloring polynomials of colored polytopes $(\Delta^4, \lambda_4),  (\Delta^3\times \Delta^1, \lambda_{31})$ and $(\Delta^2\times \Delta^2, \lambda_{22})$. By direct calculations,  we may obtain that
 the coloring polynomials of colored polytopes $(\Delta^4, \lambda_4),  (\Delta^3\times \Delta^1, \lambda_{31})$ and $(\Delta^2\times \Delta^2, \lambda_{22})$  are of the forms
$$
\begin{cases}
h_1=d\big(s_1s_2s_3s_4(s_1+s_2+s_3+s_4)\big)  \\
h_2=d\big(s_1s_2s_3s_4(s_1+s_2+s_3)+s_1s_2s_3(s_1+s_2+s_3)(s_4+s_1+as_2)\big)\\
h_3=d\big(s_1s_2(s_1+s_2)\big)d\big(s_3s_4(s_3+s_4+\varepsilon s_1)\big)
\end{cases}$$
respectively, where $a, \varepsilon=0$ or 1. Set
$\Lambda_i=\{\sigma(h_i)| \sigma\in \text{Aut}(\Hom(\Z_2, G_4))\}$.
Now, to complete the proof, it suffices to prove that

\vskip .2cm

\noindent {\bf Claim C.} {\em For each $h\in \Lambda_1\cup \Lambda_2$, $h$ is a linear combination of polynomials in $\Lambda_3$.}

 \vskip .2cm
 Up to automorphisms of $\Hom(\Z_2, G_4)$, this is equivalent to showing that $h_1$ and $h_2$ can be expressed as linear combinations of polynomials in $\Lambda_3$.
By direct calculations we have that
$h_1=h_{11}+h_{12}$
where
$$\begin{cases}
h_{11}=d\big(s_1(s_1+s_2)s_3s_4(s_1+s_2+s_3+s_4)+s_2(s_1+s_2)s_3s_4(s_1+s_2+s_3+s_4)\big)\\
h_{12}=d\big(s_1s_2(s_1+s_2)\big)d\big(s_3s_4(s_1+s_2+s_3+s_4)\big).
\end{cases}
$$
and $$h_2=h_{21}+h_{22}+
\begin{cases}
h_{23}+h_{24}  & \text{ if $a=1$}\\
h'_{23}+h'_{24}+h_{25}+h_{26} & \text{ if $a=0$}
\end{cases}$$
where
$$
\begin{cases}
h_{21}=d\big(s_2s_3(s_2+s_3)\big)d\big(s_1s_4(s_1+s_4+as_2)\big)  \\
h_{22}=d\big(s_2s_3(s_2+s_3)\big)d\Big(s_4(s_1+s_2+s_3)\big(s_4+(s_1+s_2+s_3)+as_2+s_2+s_3\big)\Big)\\
h_{23}=d\big(s_1(s_2+s_3)(s_1+s_2+s_3)\big)d\Big(s_3s_4\big(s_3+s_4+(s_1+s_2+s_3)\big)\Big)\\
h_{24}=d\big(s_1(s_2+s_3)(s_1+s_2+s_3)\big)d\Big(s_2s_4\big(s_2+s_4+s_1)\big)\Big)\\
h'_{23}=d\big(s_1(s_2+s_3)(s_1+s_2+s_3)\big)d\big(s_3s_4(s_3+s_4)\big)\\
h'_{24}=d\big(s_1(s_2+s_3)(s_1+s_2+s_3)\big)d\Big(s_2s_4\big(s_2+s_4+(s_2+s_3)\big)\Big)\\
h_{25}=d\big(s_1(s_2+s_3)(s_1+s_2+s_3)\big)d\Big(s_2(s_1+s_4)\big(s_2+(s_1+s_4)+(s_1+s_2+s_3)\big)\Big)\\
h_{26}=d\big(s_1(s_2+s_3)(s_1+s_2+s_3)\big)d\Big(s_3(s_1+s_4)\big(s_3+(s_1+s_4)+s_1\big)\Big).\\
\end{cases}$$
It is easy to check that $h_{11}\in \Lambda_2$ and $h_{12}, h_{21}, h_{22},
h_{23}, h_{24}, h'_{23}, h'_{24}, h_{25}, h_{26}\in \Lambda_3$. This proves Claim C, and thus we complete the proof of Lemma~\ref{v}.
\end{proof}

\noindent {\em Proof of Proposition~\ref{dim-4}.} We have known from
Lemma~\ref{v} that $\mathcal{V}_4^*$ is generated by those
polynomials of the form
$d\big(s_1s_2(s_1+s_2)\big)d\big(s_3s_4(s_3+s_4+\varepsilon
s_1)\big)$, each of which is the coloring polynomial of a colored
polytope $(\Delta^2\times \Delta^2, \lambda_{22})$, where $\{s_1,
s_2, s_3, s_4\}$ is a basis of $\Hom(\Z_2, G_4)$ and $\varepsilon=0$
or 1. Then it immediately follows  that $\mathcal{Z}_4(G_4)$ is
generated by the classes of small covers over $\Delta^2\times
\Delta^2$.

\vskip .1cm Now let us consider the dimension of
$\mathcal{Z}_4(G_4)$. This is equivalent to determining the
dimension of $\mathcal{V}_4^*$. By $\mathcal{W}_4$ we denote the
linear space generated by those  degree-four monomials  of
$\Z_2[\widehat{\Hom(\Z_2, G_4)}]$ whose dual monomials are all
faithful. Then $\dim_{\Z_2}\mathcal{W}_4$ is equal to the number of
all bases in $\Hom(\Z_2, G_4)\cong (\Z_2)^4$, so by \cite[Remark
2.1]{l1}, $\dim_{\Z_2}\mathcal{W}_4=840$. As a subspace of
$\mathcal{W}_4$, $\mathcal{V}_4^*$ is generated by all polynomials
of the form
$d\big(s_1s_2(s_1+s_2)\big)d\big(s_3s_4(s_3+s_4+\varepsilon
s_1)\big)$. Thus, we need to determine a maximal linearly
independent subset of  all polynomials of the form
$d\big(s_1s_2(s_1+s_2)\big)d\big(s_3s_4(s_3+s_4+\varepsilon
s_1)\big)$ in $\mathcal{W}_4$. To do this, we can give an algorithm
and write a computer program to find such a maximal linearly
independent subset, which exactly contains 510 polynomials (we shall
state the algorithm and list all elements of a basis of
$\mathcal{V}_4^*$ in Section~\ref{computer} as an appendix). Thus we
conclude that
$\dim_{\Z_2}\mathcal{V}_4^*=\dim_{\Z_2}\mathcal{Z}_4(G_4)=510$.
\hfill$\Box$

\begin{rem}\label{m3}
In a similar way to Lemma~\ref{v}, we conclude easily that
$\mathcal{V}_3^*$ is generated by the polynomials of the form
$$d\big(s_1s_2(s_1+s_2)\big)d\big(s_3(s_1+s_3)\big)$$
where $\{s_1, s_2, s_3\}$ is a basis of $\Hom(\Z_2, G_3)$. So
$\mathcal{Z}_3(G_3)$ is generated by the classes of small covers
over $\Delta^1\times\Delta^2$. We can also list a basis of
$\mathcal{V}_3^*$ as follows:
\begin{center}
\begin{tabular}{|c|c|c|c|}
\hline &&&\\
$d(x_1x_6x_7)d(x_2x_5)$ & $d(x_1x_6x_7)d(x_3x_4)$ & $d(x_2x_5x_7)d(x_3x_4)$ & $d(x_2x_5x_7)d(x_1x_6)$\\
\hline
  $d(x_3x_4x_7)d(x_2x_5)$  &  $d(x_3x_4x_7)d(x_1x_6)$ & $d(x_3x_5x_6)d(x_2x_4)$ & $d(x_3x_5x_6)d(x_1x_7)$ \\
\hline
  $d(x_2x_4x_6)d(x_1x_7)$ & $d(x_1x_4x_5)d(x_3x_6)$ & $d(x_1x_4x_5)d(x_2x_7)$ & $d(x_2x_5x_7)d(x_1x_4)$\\
\hline
$d(x_3x_4x_7)d(x_2x_6)$ &  & &\\
\hline
\end{tabular}
\end{center}
where $x_1, x_2, ..., x_{7}$ denote the 7 nontrivial elements in
$\Hom(\Z_2,G_3)$ with
$$
\begin{cases}
x_1=\rho^*_1 & x_2=\rho^*_2 \\
 x_3=\rho^*_1+\rho^*_2 & x_4=\rho^*_3\\
x_5=\rho^*_1+\rho^*_3 &x_6=\rho^*_2+\rho^*_3 \\
x_7=\rho^*_1+\rho^*_2+\rho^*_3 & \text{$\{\rho^*_1, \rho^*_2,
\rho^*_3\}$ is the standard basis of $\Hom(\Z_2,G_3)$.}
\end{cases}$$
Therefore, we obtain that $\dim_{\Z_2}\mathcal{Z}_3(G_3)=13$. 
Example~\ref{exam-basis} provides us  a nonbounding effective
$(\Z_2)^3$-action on $S^1\times{\Bbb R}P^2$ with
$\Delta^1\times\Delta^2$ as its orbit space, and applying
automorphisms of $(\Z_2)^3$ to this effective action can give all
required basis elements of $\mathcal{Z}_3(G_3)$.
\end{rem}

 \section{A summary and further problems}\label{app-2}

 Together with Theorems~\ref{main result}, \ref{color poly} and Corollaries~\ref{diff-formula}, \ref{relation}, we see that there are some essential relationships among 2-torus manifolds, coloring polynomials, colored simple convex polytopes, colored graphs, which are stated as follows:

 \begin{thm}\label{summary}
 Let $g=\sum_i t_{i,1}\cdots t_{i, n}$ be a faithful $G_n$-polynomial in $\Z_2[\widehat{\Hom(G_n, \Z_2)}]$. Then the following
 statements are all equivalent.

\begin{enumerate}
\item[(1)] $g\in \IM \phi_n$ $($i.e., there is an $n$-dimensional 2-torus manifold $M^n$ such that $\sum_{p\in M^G}[\tau_pM]=g)$;

\item[(2)] $g$ is the $G_n$-coloring polynomial of a $G_n$-colored graph $(\Gamma, \alpha)$;

\item[(3)] $g=\sum_i t_{i,1}\cdots t_{i, n}$ possesses the property that for any symmetric
polynomial function $f(x_1, ..., x_n)$ over ${\Bbb Z}_2$,
$$\sum_{i}{{f(t_{i,1}, ..., t_{i, n})}\over{t_{i,1}\cdots t_{i, n}}}\in\Z_2[\Hom(G_n, \Z_2)];$$

\item[(4)] $d(g^*)=0$;

\item[(5)] $g^*$ is the $G_n$-coloring polynomial of a $G_n$-colored simple convex polytope $(P^n, \lambda)$
\end{enumerate}
where $g^*$ is the dual polynomial of $g$.
\end{thm}

 Based upon the above equivalent results, it seems to be interesting to discuss the properties of
 regular graphs and simple convex polytopes. We see by Theorem~\ref{conj} and Proposition~\ref{small}
 that $\Gamma$ in Theorem~\ref{summary}(2) can actually be chosen as the 1-skeleton of a polytope.
 However, for a $G_n$-colored graph $(\Gamma, \alpha)$, we don't know when $\Gamma$ will become the 1-skeleton of a polytope.  Indeed, given a graph, to determine whether it is the 1-skeleton of a polytope or not
 is a quite difficult problem except for the known Steinitz theorem (see \cite{g}). In addition, Corollary~\ref{indecom} gives a sufficient condition that a simple convex polytope with a coloring is indecomposable. These observations lead us  to pose  the following problems:
\begin{enumerate}
\item[(P1)] For a $G_n$-colored graph $(\Gamma, \alpha)$, under what condition will $\Gamma$ be the 1-skeleton of a polytope?

\item[(P2)] Given a $G_n$-colored simple convex polytope $(P^n, \lambda)$, can we give a necessary and sufficient condition that $P^n$ is indecomposable?
\end{enumerate}

\begin{rem}
On the problem (P2), it is not difficult to see from the proof of
Theorem~\ref{conj} that if $P^n$ is indecomposable, then $g_{(P^n,
\lambda)}$ may not be indecomposable in $\Z_2[\widehat{\Hom(\Z_2,
G_n)}]$. However, if we add a restriction condition that the number
of all monomials of $g_{(P^n, \lambda)}$ is equal to that of all
vertices of $P^n$, then an easy argument shows that when $n=3$,
$P^3$ is indecomposable if and only if  $g_{(P^3, \lambda)}$ is
indecomposable in $\Z_2[\widehat{\Hom(\Z_2, G_3)}]$. It should be
reasonable to conjecture that this is also true in the
higher-dimensional case.
\end{rem}

\section{The reformulation of the existence theorem of tom Dieck--a simple proof of Theorem~\ref{dks}}\label{re-proof}

\subsection{The existence theorem of tom Dieck} First let us review tom Dieck's work.  Following \cite[page 216]{d}, let ${\Bbb Z}_2[U_A]$ be the polynomial ring on $2^n-1$ generators $U_A$, where $A$ runs through the non-empty subsets of $\{1,2, ..., n\}$, and each $U_A$ is regarded as a nontrivial irreducible real $G_n$-module. Let
$$L={\Bbb Z}_2[b_1, b_2, ...][[w_{(1)}, ..., w_{(n)}]]$$
be the ring of formal power series in $w_{(1)}, ..., w_{(n)}$ over the polynomial ring ${\Bbb Z}_2[b_1, b_2, ...]$ in a countable number of indeterminates $b_1, b_2, ...$. Then tom Dieck defines a ring homomorphism
$$\gamma: {\Bbb Z}_2[U_A]\longrightarrow K$$
into the quotient field $K$ of $L$ by mapping $U_A$ to
$${1\over {w_A}}(1+b_1w_A+b_2w_A^2+b_3w_A^3+\cdots)$$
where $w_A=\sum_{j\in A}w_{(j)}$. Let $M^m$ be a compact closed smooth $G_n$-manifold with fixed point set consisting only of isolated points $p_1, ...,  p_r$. Let
$$T_s=\bigoplus_AU_A^{m(A,s)}$$ be the tangential $G_n$-module at $p_s$. Let $\varphi[M^m]\in {\Bbb Z}_2[U_A]$ be the element
$$\sum_{s=1}^r\prod_AU_A^{m(A,s)}, $$
which is called {\em geometric}. The geometric elements form a subring of  ${\Bbb Z}_2[U_A]$. Now the existence theorem of tom Dieck may be stated as follows.
 \begin{thm} [tom Dieck~\cite{d}]\label{dieck}  An element $x\in {\Bbb Z}_2[U_A]$ is geometric if and only if $\gamma (x)$ is contained in $L$.
 \end{thm}

\subsection{A simple proof of Theorem~\ref{dks}}
Now let us show how to get our Theorem 2.2 from the existence theorem of tom Dieck.
\vskip .2cm
Since each $U_A$ is regarded as a non-trivial irreducible real
$G_n$-modules,  we see that ${\Bbb Z}_2[U_A, \emptyset]$ is exactly identified with $\mathcal{R}_*(G_n)=\sum_{m\geq 0}\mathcal{R}_m(G_n)$. Therefore,  the above map $\varphi$ is actually the Stong monomorphism $$\phi_*: \mathcal{Z}_*(G_n)\longrightarrow \mathcal{R}_*(G_n).$$ On the other hand,
we see that the $w_A$ actually corresponds to the  equivariant Euler class of the  $U_A$, so it belongs to $H^1(BG_n; {\Bbb Z}_2)$. Thus,
$$L={\Bbb Z}_2[b_1, b_2, ...][[w_{(1)}, ..., w_{(n)}]]$$
can be regarded as
$$H^*(BG_n; {\Bbb Z}_2)[b_1, b_2, ...].$$
Furthermore, $$\gamma: {\Bbb Z}_2[U_A]\longrightarrow K$$
can be regarded as
$$\gamma: \mathcal{R}_*(G_n)\longrightarrow S^{-1}H^*(BG_n; {\Bbb Z}_2)[b_1, b_2, ...]$$
where $S=H^1(BG_n;{\Bbb Z}_2)\setminus\{0\}$.
Now, let $\tau_1+\cdots+\tau_r$ be an element in $\mathcal{R}_m(G_n)$. Then by Theorem~\ref{dieck}, we have that  $x=\tau_1+\cdots+\tau_r\in \text{Im} \phi_m$ if and only if $\gamma(x)\in H^*(BG_n; {\Bbb Z}_2)[b_1, b_2, ...].$

\vskip .2cm
Write $\tau_s=\prod_AU_A^{m(A,s)}$ which is $m$-dimenional. Then $\tau_1+\cdots+\tau_r=\sum_{s=1}^r\prod_AU_A^{m(A,s)}$.
We note that $\gamma$ is a ring homomorphism. Thus
$$\gamma(\tau_1+\cdots+\tau_r)=\sum_{s=1}^r\prod_A\gamma(U_A^{m(A,s)}).$$
For each $\tau_s$, write $\chi^{G_n}(\tau_s)=t_{s, 1}\cdots t_{s, m}$ (i.e.,  the equivariant Euler class of $\tau_s$),  where $t_{s, j}\in H^1(BG_n;{\Bbb Z}_2)$.
Then we have that
\begin{align*}
\gamma(\tau_s) &=\prod_{j=1}^m{1\over {t_{s,j}}}(1+t_{s,j}b_1+t_{s,j}^2b_2+t_{s,j}^3b_3+\cdots)\\
&={1\over {t_{s, 1}\cdots t_{s, m}}}\{1+(t_{s, 1}+\cdots+ t_{s, m})b_1+(t_{s,1}t_{s,2}+\cdots+t_{s, m-1}t_{s, m})b_1^2\\
&\ \ +(t^2_{s, 1}+\cdots+ t^2_{s, m})b_2+\cdots\}\\
&={1\over {\chi^{({\Bbb Z}_2)^n}(\tau_s)}}\sum_{\omega=(i_1, ..., i_u) \atop u\leq m}S_\omega(t_{s,1},...,t_{s,m})b_\omega  \end{align*}
where $\omega=(i_1, ..., i_u)$ is a partition of $|\omega|=i_1+\cdots+i_u$,
$S_\omega(t_{s,1},...,t_{s,m})=\sum t_{s,1}^{i_1}\cdots t_{s,u}^{i_u}$ denotes the usual smallest symmetric polynomial containing the given monomial, and $b_\omega=\prod_{j=1}^ub_{i_j}$. Therefore,
\begin{align*}
\gamma(x) &=\gamma(\tau_1+\cdots+\tau_r)\\
&=\sum_{s=1}^r{1\over {\chi^{({\Bbb Z}_2)^n}(\tau_s)}}\sum_{\omega=(i_1, ..., i_u) \atop u\leq m}S_\omega(t_{s,1},...,t_{s,m})b_\omega.  \end{align*}
This induces that $\gamma(x)\in H^*(BG_n; {\Bbb Z}_2)[b_1, b_2, ...]$ if and only if for all partitions $\omega=(i_1, ..., i_u)$ with $u\leq m$,
$$\sum_{s=1}^r{{S_\omega(t_{s,1},...,t_{s,m})}\over {\chi^{G_n}(\tau_s)}}\in H^*(BG_n; {\Bbb Z}_2).$$
It is well-known (see also \cite{ms1}) that the polynomial
sub-algebra of ${\Bbb Z}_2[x_1, ...,  x_m]$ generated by all
$S_\omega(x_1, ..., x_m), |\omega|\geq 0$ is identified with the
sub-algebra generated by all elementary symmetric polynomial
functions $\sigma_i(x_1, ..., x_m),  0\leq i\leq m$. It follows that
$\gamma(x)\in H^*(BG_n; {\Bbb Z}_2)[b_1, b_2, ...]$ if and only if
for all symmetric polynomial functions $f(x_1, ..., x_m)$ over
${\Bbb Z}_2$,
$$\sum_{s=1}^r{{f(t_{s,1},...,t_{s,m})}\over {\chi^{G_n}(\tau_s)}}\in H^*(BG_n; {\Bbb Z}_2).$$
Combining with the above arguments, we obtain the required statement of  Theorem 2.2,
 which says that $\tau_1+\cdots+\tau_r\in \text{Im} \phi_m$ if and only if for all symmetric polynomial functions $f(x_1, ..., x_m)$ over ${\Bbb Z}_2$,
$$\sum_{s=1}^r{{f(t_{s,1},...,t_{s,m})}\over {\chi^{G_n}(\tau_s)}}\in H^*(BG_n; {\Bbb Z}_2).$$
This completes the proof.
\hfill$\Box$

\section{Appendix--An algorithm to determine a basis  of $\mathcal{V}_4^*$}\label{computer}

Here we introduce an algorithm how to produce a basis of
$\mathcal{V}_4^*$. As shown in the proof of Proposition~\ref{dim-4},
to get a basis of $\mathcal{V}_4^*$, we need to determine a maximal
linearly independent subset of  those polynomials of two types
$f_1=d\big(s_1s_2(s_1+s_2)\big)d\big(s_3s_4(s_3+s_4)\big)$ and
$f_2=d\big(s_1s_2(s_1+s_2)\big)d\big(s_3s_4(s_1+s_3+s_4)\big)$ in
$\mathcal{W}_4$.

\vskip .1cm Regarded each monomial of $\mathcal{W}_4$ as a basis of
$(\Z_2)^4$ (as a 4-dimensional linear space over $\Z_2$). Then
$\mathcal{W}_4$ is a 840-dimensional linear space generated by all
bases of $(\Z_2)^4$, and each polynomial of type $f_i$ is a linear
combination of 9 bases of $(\Z_2)^4$.
Now our algorithm is divided into the following four steps:
\begin{enumerate}
\item[Step 1.]  Fixed an ordered basis of  $\mathcal{W}_4$, which can be written as a $4\times (840\times 4)$-matrix $A$ such that for each $i=1, 2, ..., 840$, the collection of the $(4i-3)$-th, $(4i-2)$-th, $(4i-1)$-th, $(4i)$-th columns  in $A$ forms a basis of $(\Z_2)^4$.

\item[Step 2.] Compute
out 840-tuple coordinate vectors  of all polynomials of types $f_1$
and $f_2$ relative to this ordered basis  of $\mathcal{W}_4$, where each coordinate vector has
only 9 nonzero coordinates 1. Then we use these coordinate vectors as rows
to produce a matrix $B$.
\item[Step 3.] Do elementary column operations on $B$ to determine the rank of $B$ which is exactly the dimension of $\mathcal{V}_4^*$, and choose a maximal
linearly independent subset from all rows of $B$, which gives a matrix $L$.
\item[Step 4.] Compare the matrix $A$ and the  matrix $L$, we can read out the required basis of $\mathcal{V}_4^*$.
\end{enumerate}

Based upon this algorithm, we can write a computer program to find a basis of $\mathcal{V}_4^*$, which exactly consists of 510 polynomials,  listed as follows: {\tiny

\begin{center}
\begin{tabular}{|c|c|c|c|c|c|}
\hline &&&&&\\
$1$ & $(1;14; 15; 6; 11; 13)$ & $(1;14; 15; 7; 10; 13)$ & $(1; 14; 15; 4; 9; 13)$ & $(1; 14; 15; 5; 8; 13)$ & $(1; 14; 15; 7; 11; 12)$\\
\hline
$2$ & $(1; 14; 15; 6; 10; 12)$ & $(1; 14; 15; 5; 9; 12)$ & $(1; 14; 15; 4; 8; 12)$ & $(1; 14; 15; 2; 9; 11)$ & $(1; 14; 15; 3; 8; 11)$ \\
\hline $3$ &
$(1; 14; 15; 3; 9; 10)$ & $(1; 14; 15; 2; 8; 10)$ & $(1; 14; 15; 2; 5; 7)$ & $(1; 14; 15; 3; 4; 7)$ & $(1; 14; 15; 3; 5; 6)$\\
\hline
$4$ & $(1; 14; 15; 2; 4; 6)$ & $(2; 13; 15; 7; 11; 12)$ & $(2; 13; 15; 6; 10; 12)$ & $(2; 13; 15; 5; 9; 12)$ & $(2; 13; 15; 4; 8; 12)$\\
\hline
$5$ & $(1; 10; 11; 2; 13; 15)$ & $(2; 13; 15; 3; 8; 11)$ & $(2; 13; 15; 5; 11; 14)$ & $(2; 13; 15; 3; 9; 10)$ & $(2; 13; 15; 4; 10; 14)$\\
\hline
$6$ & $(1; 8; 9; 2; 13; 15)$ & $(2; 13; 15; 7; 9; 14)$ & $(2; 13; 15; 6; 8; 14)$ & $(1; 6; 7; 2; 13; 15)$ & $(2; 13; 15; 3; 4; 7)$\\
\hline
$7$ & $(2; 13; 15; 3; 5; 6)$ & $(1; 4; 5; 2; 13; 15)$ & $(1; 10; 11; 3; 12; 15)$ & $(2; 9; 11; 3; 12; 15)$ & $(3; 12; 15; 6; 11; 13)$\\
\hline
$8$ & $(3; 12; 15; 5; 11; 14)$ & $(2; 8; 10; 3; 12; 15)$ & $(3; 12; 15; 7; 10; 13)$ & $(3; 12; 15; 4; 10; 14)$ & $(1; 8; 9; 3; 12; 15)$\\
\hline
$9$ & $(3; 12; 15; 7; 9; 14)$ & $(3; 12; 15; 4; 9; 13)$ & $(3; 12; 15; 6; 8; 14)$ & $(3; 12; 15; 5; 8; 13)$ & $(1; 6; 7; 3; 12; 15)$\\
\hline
$10$ & $(2; 5; 7; 3; 12; 15)$ & $(2; 4; 6; 3; 12; 15)$ & $(1; 4; 5; 3; 12; 15)$ & $(3; 9; 10; 4; 11; 15)$ & $(2; 8; 10; 4; 11; 15)$\\
\hline
$11$ & $(4; 11; 15; 7; 10; 13)$ & $(4; 11; 15; 6; 10; 12)$ & $(1; 8; 9; 4; 11; 15)$ & $(4; 11; 15; 7; 9; 14)$ & $(4; 11; 15; 5; 9; 12)$\\
\hline
$12$ & $(4; 11; 15; 6; 8; 14)$ & $(4; 11; 15; 5; 8; 13)$ & $(1; 6; 7; 4; 11; 15)$ & $(2; 5; 7; 4; 11; 15)$ & $(3; 5; 6; 4; 11; 15)$\\
\hline
$13$ & $(1; 2; 3; 4; 11; 15)$ & $(1; 8; 9; 5; 10; 15)$ & $(5; 10; 15; 7; 9; 14)$ & $(4; 9; 13; 5; 10; 15)$ &$(2; 9; 11; 5; 10; 15)$\\
\hline
$14$ & $(5; 10; 15; 6; 8; 14)$ & $(4; 8; 12; 5; 10; 15)$ & $(3; 8; 11; 5; 10; 15)$ & $(1; 6; 7; 5; 10; 15)$ &$(3; 4; 7; 5; 10; 15)$\\
\hline
$15$ & $(2; 4; 6; 5; 10; 15)$ & $(1; 2; 3; 5; 10; 15)$ &$(5; 8; 13; 6; 9; 15)$ &$(4; 8; 12; 6; 9; 15)$ &$(3; 8; 11; 6; 9; 15)$\\
\hline
$16$ & $(2; 8; 10; 6; 9; 15)$ &$(2; 5; 7; 6; 9; 15)$ &$(3; 4; 7; 6; 9; 15)$ &$(1; 4; 5; 6; 9; 15)$ &$(1; 2; 3; 6; 9; 15)$\\
\hline
$17$ & $(3; 5; 6; 7; 8; 15)$ &$(2; 4; 6; 7; 8; 15)$ &$(1; 4; 5; 7; 8; 15)$ &$(1; 2; 3; 7; 8; 15)$ &$(3; 13; 14; 7; 11; 12)$\\
\hline
$18$ & $(3; 13; 14; 6; 10; 12)$ &$(3; 13; 14; 5; 9; 12)$ &$(3; 13; 14; 4; 8; 12)$ &$(1; 10; 11; 3; 13; 14)$ &$(2; 9; 11; 3; 13; 14)$\\
\hline
$19$ & $(3; 13; 14; 4; 11; 15)$ &$(2; 8; 10; 3; 13; 14)$ &$(3; 13; 14; 5; 10; 15)$ &$(1; 8; 9; 3; 13; 14)$ &$(3; 13; 14; 6; 9; 15)$\\
\hline
$20$ & $(3; 13; 14; 7; 8; 15)$ &$(1; 6; 7; 3; 13; 14)$ & $(2; 5; 7; 3; 13; 14)$ &$(2; 4; 6; 3; 13; 14)$ &$(1; 4; 5; 3; 13; 14)$\\
\hline
$21$ & $(1; 10; 11; 2; 12; 14)$ &$(2; 12; 14; 3; 8; 11)$ &$(2; 12; 14; 6; 11; 13)$ &$(2; 12; 14; 4; 11; 15)$ &$(2; 12; 14; 3; 9; 10)$\\
\hline
$22$ & $(2; 12; 14; 7; 10; 13)$ &$(2; 12; 14; 5; 10; 15)$ &$(1; 8; 9; 2; 12; 14)$ &$(2; 12; 14; 6; 9; 15)$ &$(2; 12; 14; 4; 9; 13)$\\
\hline
$23$ & $(2; 12; 14; 7; 8; 15)$ &$(2; 12; 14; 5; 8; 13)$ &$(1; 6; 7; 2; 12; 14)$ &$(2; 12; 14; 3; 4; 7)$ &$(2; 12; 14; 3; 5; 6)$\\
\hline
$24$ & $(1; 4; 5; 2; 12; 14)$ &$(3; 9; 10; 5; 11; 14)$ &$(2; 8; 10; 5; 11; 14)$ &$(5; 11; 14; 7; 10; 13)$ &$(5; 11; 14; 6; 10; 12)$\\
\hline
$25$ & $(1; 8; 9; 5; 11; 14)$ &$(5; 11; 14; 6; 9; 15)$ &$(4; 9; 13; 5; 11; 14)$ &$(5; 11; 14; 7; 8; 15)$ &$(4; 8; 12; 5; 11; 14)$\\
\hline
$26$ & $(1; 6; 7; 5; 11; 14)$ &$(3; 4; 7; 5; 11; 14)$ &$(2; 4; 6; 5; 11; 14)$ &$(1; 2; 3; 5; 11; 14)$ &$(1; 8; 9; 4; 10; 14)$\\
\hline
$27$ & $(4; 10; 14; 6; 9; 15)$ &$(4; 10; 14; 5; 9; 12)$ &$(2; 9; 11; 4; 10; 14)$ &$(4; 10; 14; 7; 8; 15)$ &$(4; 10; 14; 5; 8; 13)$\\
\hline
$28$ & $(3; 8; 11; 4; 10; 14)$ &$(1; 6; 7; 4; 10; 14)$ &$(2; 5; 7; 4; 10; 14)$ &$(3; 5; 6; 4; 10; 14)$ &$(1; 2; 3; 4; 10; 14)$\\
\hline
$29$ & $(5; 8; 13; 7; 9; 14)$ &$(4; 8; 12; 7; 9; 14)$ &$(3; 8; 11; 7; 9; 14)$ &$(2; 8; 10; 7; 9; 14)$ &$(3; 5; 6; 7; 9; 14)$\\
\hline
$30$ & $(2; 4; 6; 7; 9; 14)$ &$(1; 4; 5; 7; 9; 14)$ &$(1; 2; 3; 7; 9; 14)$ &$(2; 5; 7; 6; 8; 14)$ &$(3; 4; 7; 6; 8; 14)$\\
\hline
$31$ & $(1; 4; 5; 6; 8; 14)$ &$(1; 2; 3; 6; 8; 14)$ &$(1; 12; 13; 2; 9; 11)$ &$(1; 12; 13; 3; 8; 11)$ &$(1; 12; 13; 5; 11; 14)$\\
\hline
$32$ & $(1; 12; 13; 4; 11; 15)$ &$(1; 12; 13; 3; 9; 10)$ &$(1; 12; 13; 2; 8; 10)$ &$(1; 12; 13; 5; 10; 15)$ &$(1; 12; 13; 4; 10; 14)$\\
\hline
$33$ & $(1; 12; 13; 7; 9; 14)$ &$(1; 12; 13; 6; 9; 15)$ &$(1; 12; 13; 7; 8; 15)$ &$(1; 12; 13; 6; 8; 14)$ &$(1; 12; 13; 2; 5; 7)$\\
\hline
$34$ & $(1; 12; 13; 3; 4; 7)$ &$(1; 12; 13; 3; 5; 6)$ &$(1; 12; 13; 2; 4; 6)$ &$(3; 9; 10; 6; 11; 13)$ &$(2; 8; 10; 6; 11; 13)$\\
\hline
$35$ & $(5; 10; 15; 6; 11; 13)$ &$(4; 10; 14; 6; 11; 13)$ &$(1; 8; 9; 6; 11; 13)$ &$(6; 11; 13; 7; 9; 14)$ &$(5; 9; 12; 6; 11; 13)$\\
\hline
$36$ & $(6; 11; 13; 7; 8; 15)$ &$(4; 8; 12; 6; 11; 13)$ &$(2; 5; 7; 6; 11; 13)$ &$(3; 4; 7; 6; 11; 13)$ &$(1; 4; 5; 6; 11; 13)$\\
\hline
$37$ & $(1; 2; 3; 6; 11; 13)$ &$(1; 8; 9; 7; 10; 13)$ &$(6; 9; 15; 7; 10; 13)$ &$(5; 9; 12; 7; 10; 13)$ &$(2; 9; 11; 7; 10; 13)$\\
\hline
$38$ & $(6; 8; 14; 7; 10; 13)$ &$(4; 8; 12; 7; 10; 13)$ &$(3; 8; 11; 7; 10; 13)$ &$(3; 5; 6; 7; 10; 13)$ &$(2; 4; 6; 7; 10; 13)$\\
\hline
$39$ & $(1; 4; 5; 7; 10; 13)$ &$(1; 2; 3; 7; 10; 13)$ &$(4; 9; 13; 7; 8; 15)$ &$(4; 9; 13; 6; 8; 14)$ &$(3; 8; 11; 4; 9; 13)$\\
\hline
$40$ & $(2; 8; 10; 4; 9; 13)$ &$(1; 6; 7; 4; 9; 13)$ &$(2; 5; 7; 4; 9; 13)$ &$(3; 5; 6; 4; 9; 13)$ &$(1; 2; 3; 4; 9; 13)$\\
\hline
$41$ & $(1; 6; 7; 5; 8; 13)$ &$(3; 4; 7; 5; 8; 13)$ &$(2; 4; 6; 5; 8; 13)$ &$(1; 2; 3; 5; 8; 13)$ &$(3; 9; 10; 7; 11; 12)$\\
\hline
$42$ & $(2; 8; 10; 7; 11; 12)$ &$(5; 10; 15; 7; 11; 12)$ &$(4; 10; 14; 7; 11; 12)$ &$(1; 8; 9; 7; 11; 12)$ &$(6; 9; 15; 7; 11; 12)$\\
\hline
$43$ & $(4; 9; 13; 7; 11; 12)$ &$(6; 8; 14; 7; 11; 12)$ &$(5; 8; 13; 7; 11; 12)$ &$(3; 5; 6; 7; 11; 12)$ &$(2; 4; 6; 7; 11; 12)$\\
\hline
$44$ & $(1; 4; 5; 7; 11; 12)$ &$(1; 2; 3; 7; 11; 12)$ &$(1; 8; 9; 6; 10; 12)$ &$(6; 10; 12; 7; 9; 14)$ &$(4; 9; 13; 6; 10; 12)$\\
\hline
$45$ & $(2; 9; 11; 6; 10; 12)$ &$(6; 10; 12; 7; 8; 15)$ &$(5; 8; 13; 6; 10; 12)$ &$(3; 8; 11; 6; 10; 12)$ &$(2; 5; 7; 6; 10; 12)$\\
\hline
$46$ & $(1; 4; 5; 6; 10; 12)$ &$(1; 2; 3; 6; 10; 12)$ &$(5; 9; 12; 7; 8; 15)$ &$(3; 8; 11; 5; 9; 12)$ &$(2; 8; 10; 5; 9; 12)$\\
\hline
$47$ & $(1; 6; 7; 5; 9; 12)$ &$(1; 2; 3; 5; 9; 12)$ &$(1; 6; 7; 4; 8; 12)$ &$(1; 2; 3; 4; 8; 12)$ &$(1; 10; 11; 7; 9; 14)$\\
\hline
$48$ & $(1; 10; 11; 6; 9; 15)$ &$(1; 10; 11; 5; 9; 12)$ &$(1; 10; 11; 4; 9; 13)$ &$(1; 10; 11; 7; 8; 15)$ &$(1; 10; 11; 6; 8; 14)$\\
\hline
$49$ & $(1; 10; 11; 5; 8; 13)$ &$(1; 10; 11; 4; 8; 12)$ &$(1; 10; 11; 2; 5; 7)$ &$(1; 10; 11; 3; 4; 7)$ &$(1; 10; 11; 3; 5; 6)$\\
\hline
$50$ & $(1; 10; 11; 2; 4; 6)$ &$(2; 9; 11; 7; 8; 15)$ &$(2; 9; 11; 6; 8; 14)$ &$(1; 6; 7; 2; 9; 11)$ &$(2; 9; 11; 3; 4; 7)$\\
\hline
$51$ & $(2; 9; 11; 3; 5; 6)$ &$(1; 4; 5; 2; 9; 11)$ &$(1; 6; 7; 3; 8; 11)$ &$(2; 5; 7; 3; 8; 11)$ &$(2; 4; 6; 3; 8; 11)$\\
\hline
$52$ & $(1; 4; 5; 3; 8; 11)$ &$(3; 9; 10; 7; 8; 15)$ &$(1; 6; 7; 3; 9; 10)$ &$(1; 6; 7; 2; 8; 10)$ &$(1; 8; 9; 2; 5; 7)$\\
\hline
$53$ & $(1; 14; 15; 9; 11; 13)$ &$(1; 14; 15; 8; 10; 13)$ &$(1; 14; 15; 5; 7; 13)$ &$(1; 14; 15; 4; 6; 13)$ &$(1; 14; 15; 8; 11; 12)$\\
\hline
$54$ & $(1; 14; 15; 9; 10; 12)$ &$(1; 14; 15; 4; 7; 12)$ &$(1; 14; 15; 5; 6; 12)$ &$(1; 14; 15; 3; 7; 11)$ &$(1; 14; 15; 2; 6; 11)$\\
\hline
$55$ & $(1; 14; 15; 2; 7; 10)$ &$(1; 14; 15; 3; 6; 10)$ &$(2; 13; 15; 8 11 12)$ &$(2; 13; 15; 9; 10; 12)$ &$(2; 13; 15; 4; 7; 12)$\\
\hline
$56$ & $(2; 13; 15; 5; 6; 12)$ &$(2; 13; 15; 10; 11; 14)$ &$(2; 13; 15; 3; 7; 11)$ &$(2; 13; 15; 1; 5; 11)$ &$(2; 13; 15; 3; 6; 10)$\\
\hline
$57$ & $(2; 13; 15; 1; 4; 10)$ &$(2; 13; 15; 8; 9; 14)$ &$(2; 13; 15; 1; 7; 9)$ &$(2; 13; 15; 1; 6; 8)$ &$(3; 12; 15; 10; 11; 14)$\\
\hline
$58$ & $(3; 12; 15; 9; 11; 13)$ &$(3; 12; 15; 2; 6; 11)$ &$(3; 12; 15; 1; 5; 11)$ &$(3; 12; 15; 8; 10; 13)$ &$(3; 12; 15; 2; 7; 10)$\\
\hline
$59$ & $(3; 12; 15; 1; 4; 10)$ &$(3; 12; 15; 8; 9; 14)$ &$(3; 12; 15; 1; 7; 9)$ &$(3; 12; 15; 2; 4; 9)$ &$(3; 12; 15; 1; 6; 8)$\\
\hline
$60$ & $(3; 12; 15; 2; 5; 8)$ &$(4; 11; 15; 9; 10; 12)$ &$(4; 11; 15; 8; 10; 13)$ &$(4; 11; 15; 2; 7; 10)$ &$(4; 11; 15; 3; 6; 10)$\\
\hline
$61$ & $(4; 11; 15; 8; 9; 14)$ &$(4; 11; 15; 1; 7; 9)$ &$(4; 11; 15; 3; 5; 9)$ &$(4; 11; 15; 1; 6; 8)$ &$(4; 11; 15; 2; 5; 8)$\\
\hline
$62$ & $(4; 11; 15; 2; 3; 14)$ &$(4; 11; 15; 1; 3; 13)$ &$(4; 11; 15; 1; 2; 12)$ &$(5; 10; 15; 8; 9; 14)$ &$(5; 10; 15; 1; 7; 9)$\\
\hline
$63$ & $(5; 10; 15; 2; 4; 9)$ &$(5; 10; 15; 1; 6; 8)$ &$(5; 10; 15; 3; 4; 8)$ &$(5; 10; 15; 4; 7; 12)$ &$(5; 10; 15; 3; 7; 11)$\\
\hline
$64$ & $(5; 10; 15; 4; 6; 13)$ &$(5; 10; 15; 2; 6; 11)$ &$(5; 10; 15; 2; 3; 14)$ &$(5; 10; 15; 1; 3; 13)$ &$(5; 10; 15; 1; 2; 12)$\\
\hline
$65$ & $(6; 9; 15; 2; 5; 8)$ &$(6; 9; 15; 3; 4; 8)$ &$(6; 9; 15; 5; 7; 13)$ &$(6; 9; 15; 4; 7; 12)$ &$(6; 9; 15; 3; 7; 11)$\\
\hline
$66$ & $(6; 9; 15; 2; 7; 10)$ &$(6; 9; 15; 4; 5; 14)$ &$(6; 9; 15; 1; 5; 11)$ &$(6; 9; 15; 1; 4; 10)$ &$(6; 9; 15; 2; 3; 14)$\\
\hline
$67$ & $(6; 9; 15; 1; 3; 13)$ &$(6; 9; 15; 1; 2; 12)$ &$(7; 8; 15; 5; 6; 12)$ &$(7; 8; 15; 4; 6; 13)$ &$(7; 8; 15; 3; 6; 10)$\\
\hline
$68$ & $(7; 8; 15; 2; 6; 11)$ &$(7; 8; 15; 4; 5; 14)$ &$(7; 8; 15; 3; 5; 9)$ &$(7; 8; 15; 1; 5; 11)$ &$(7; 8; 15; 2; 4; 9)$\\
\hline
$69$ & $(7; 8; 15; 1; 4; 10)$ &$(7; 8; 15; 2; 3; 14)$ &$(7; 8; 15; 1; 3; 13)$ &$(7; 8; 15; 1; 2; 12)$ &$(3; 13; 14; 9; 11; 12)$\\
\hline
$70$ & $(3; 13; 14; 8; 10; 12)$ &$(3; 13; 14; 5; 7; 12)$ &$(3; 13; 14; 4; 6; 12)$ &$(3; 13; 14; 10; 11; 15)$ &$(3; 13; 14; 2; 7; 11)$\\
\hline
\end{tabular}
\end{center}

\begin{center}
\begin{tabular}{|c|c|c|c|c|c|}
\hline &&&&&\\
$71$ & $(3; 13; 14; 1; 4; 11)$ &$(3; 13; 14; 2; 6; 10)$ &$(3; 13; 14; 1; 5; 10)$ &$(3; 13; 14; 8; 9; 15)$ &$(3; 13; 14; 1; 6; 9)$\\
\hline
$72$ & $(3; 13; 14; 1; 7; 8)$ &$(2; 12; 14; 10; 11; 15)$ &$(2; 12; 14; 8; 11; 13)$ &$(2; 12; 14; 3; 6; 11)$ &$(2; 12; 14; 1; 4; 11)$\\
\hline
$73$ & $(2; 12; 14; 1; 5; 10)$ &$(2; 12; 14; 8; 9; 15)$ &$(2; 12; 14; 1; 6; 9)$ &$(2; 12; 14; 1; 7; 8)$ &$(5; 11; 14; 9; 10; 13)$\\
\hline
$74$ & $(5; 11; 14; 8; 10; 12)$ &$(5; 11; 14; 3; 7; 10)$ &$(5; 11; 14; 2; 6; 10)$ &$(5; 11; 14; 8; 9; 15)$ &$(5; 11; 14; 1; 6; 9)$\\
\hline
$75$ & $(5; 11; 14; 3; 4; 9)$ &$(5; 11; 14; 1; 7; 8)$ &$(5; 11; 14; 2; 4; 8)$ &$(5; 11; 14; 2; 3; 15)$ &$(5; 11; 14; 1; 3; 12)$\\
\hline
$76$ & $(5; 11; 14; 1; 2; 13)$ &$(4; 10; 14; 8; 9; 15)$ &$(4; 10; 14; 1; 6; 9)$ &$(4; 10; 14; 2; 5; 9)$ &$(4; 10; 14; 1; 7; 8)$\\
\hline
$77$ & $(4; 10; 14; 2; 7; 11)$ &$(4; 10; 14; 2; 3; 15)$ &$(4; 10; 14; 1; 3; 12)$ &$(4; 10; 14; 1; 2; 13)$ &$(7; 9; 14; 3; 5; 8)$\\
\hline
$78$ & $(7; 9; 14; 2; 4; 8)$ &$(7; 9; 14; 5; 6; 13)$ &$(7; 9; 14; 4; 6; 12)$ &$(7; 9; 14; 3; 6; 11)$ &$(7; 9; 14; 2; 6; 10)$\\
\hline
$79$ & $(7; 9; 14; 4; 5; 15)$ &$(7; 9; 14; 1; 5; 10)$ &$(7; 9; 14; 1; 4; 11)$ &$(7; 9; 14; 2; 3; 15)$ &$(7; 9; 14; 1; 3; 12)$\\
\hline
$80$ & $(7; 9; 14; 1; 2; 13)$ &$(6; 8; 14; 5; 7; 12)$ &$(6; 8; 14; 4; 5; 15)$ &$(6; 8; 14; 2; 5; 9)$ &$(6; 8; 14; 1; 5; 10)$\\
\hline
$81$ & $(6; 8; 14; 1; 4; 11)$ &$(6; 8; 14; 2; 3; 15)$ &$(6; 8; 14; 1; 3; 12)$ &$(6; 8; 14; 1; 2; 13)$ &$(1; 12; 13; 9; 11; 15)$\\
\hline
$82$ & $(1; 12; 13; 8; 11; 14)$ &$(1; 12; 13; 3; 5; 11)$ &$(1; 12; 13; 2; 4; 11)$ &$(1; 12; 13; 9; 10; 14)$ &$(1; 12; 13; 8; 10; 15)$\\
\hline
$83$ & $(1; 12; 13; 2; 5; 10)$ &$(1; 12; 13; 3; 4; 10)$ &$(1; 12; 13; 3; 7; 9)$ &$(1; 12; 13; 2; 6; 9)$ &$(1; 12; 13; 2; 7; 8)$\\
\hline
$84$ & $(1; 12; 13; 3; 6; 8)$ &$(6; 11; 13; 9; 10; 14)$ &$(6; 11; 13; 8; 10; 15)$ &$(6; 11; 13; 2; 5; 10)$ &$(6; 11; 13; 3; 4; 10)$\\
\hline
$85$ & $(6; 11; 13; 8; 9; 12)$ &$(6; 11; 13; 3; 7; 9)$ &$(6; 11; 13; 1; 5; 9)$ &$(6; 11; 13; 2; 7; 8)$ &$(6; 11; 13; 1; 4; 8)$\\
\hline
$86$ & $(6; 11; 13; 2; 3; 12)$ &$(6; 11; 13; 1; 3; 15)$ &$(6; 11; 13; 1; 2; 14)$ &$(7; 10; 13; 8; 9; 12)$ &$(7; 10; 13; 2; 6; 9)$\\
\hline
$87$ & $(7; 10; 13; 1; 5; 9)$ &$(7; 10; 13; 3; 6; 8)$ &$(7; 10; 13; 1; 4; 8)$ &$(7; 10; 13; 5; 6; 14)$ &$(7; 10; 13; 4; 6; 15)$\\
\hline
$88$ & $(7; 10; 13; 3; 5; 11)$ &$(7; 10; 13; 2; 4; 11)$ &$(7; 10; 13; 2; 3; 12)$ &$(7; 10; 13; 1; 3; 15)$ &$(7; 10; 13; 1; 2; 14)$\\
\hline
$89$ & $(4; 9; 13; 2; 7; 8)$ &$(4; 9; 13; 2; 3; 12)$ &$(4; 9; 13; 1; 3; 15)$ &$(4; 9; 13; 1; 2; 14)$ &$(5; 8; 13; 6; 7; 12)$\\
\hline
$90$ & $(5; 8; 13; 1; 3; 15)$ &$(5; 8; 13; 1; 2; 14)$ &$(7; 11; 12; 9; 10; 15)$ &$(7; 11; 12; 8; 10; 14)$ &$(7; 11; 12; 3; 5; 10)$\\
\hline
$91$ & $(7; 11; 12; 2; 4; 10)$ &$(7; 11; 12; 8; 9; 13)$ &$(7; 11; 12; 3; 6; 9)$ &$(7; 11; 12; 1; 4; 9)$ &$(7; 11; 12; 2; 6; 8)$\\
\hline
$92$ & $(7; 11; 12; 1; 5; 8)$ &$(7; 11; 12; 2; 3; 13)$ &$(7; 11; 12; 1; 3; 14)$ &$(7; 11; 12; 1; 2; 15)$ &$(6; 10; 12; 8; 9; 13)$\\
\hline
$93$ & $(6; 10; 12; 2; 7; 9)$ &$(6; 10; 12; 1; 4; 9)$ &$(6; 10; 12; 2; 5; 11)$ &$(6; 10; 12; 2; 3; 13)$ &$(6; 10; 12; 1; 3; 14)$\\
\hline
$94$ & $(6; 10; 12; 1; 2; 15)$ &$(5; 9; 12; 2; 3; 13)$ &$(5; 9; 12; 1; 3; 14)$ &$(5; 9; 12; 1; 2; 15)$ &$(1; 10; 11; 5; 7; 9)$\\
\hline
$95$ & $(1; 10; 11; 4; 6; 9)$ &$(1; 10; 11; 4; 7; 8)$ &$(1; 10; 11; 5; 6; 8)$ &$(1; 10; 11; 3; 7; 15)$ &$(1; 10; 11; 2; 7; 14)$\\
\hline
$96$ & $(1; 10; 11; 3; 6; 14)$ &$(1; 10; 11; 2; 6; 15)$ &$(1; 10; 11; 3; 5; 13)$ &$(1; 10; 11; 2; 5; 12)$ &$(1; 10; 11; 3; 4; 12)$\\
\hline
$97$ & $(1; 10; 11; 2; 4; 13)$ &$(2; 9; 11; 6; 7; 10)$ &$(2; 9; 11; 1; 7; 13)$ &$(2; 9; 11; 3; 6; 14)$ &$(2; 9; 11; 1; 6; 12)$\\
\hline
$98$ & $(2; 9; 11; 4; 5; 10)$ &$(2; 9; 11; 3; 5; 13)$ &$(2; 9; 11; 1; 5; 15)$ &$(2; 9; 11; 3; 4; 12)$ &$(2; 9; 11; 1; 4; 14)$\\
\hline
$99$ & $(3; 8; 11; 2; 6; 15)$ &$(3; 8; 11; 1; 5; 15)$ &$(3; 8; 11; 2; 4; 13)$ &$(3; 8; 11; 1; 4; 14)$ &$(7; 11; 12; 5; 6; 8)$\\
\hline
$100$ & $(6; 11; 13; 4; 5; 10)$ &$(5; 11; 14; 3; 4; 12)$ &$(5; 11; 14; 1; 3; 9)$ &$(4; 11; 15; 2; 3; 10)$ &$(4; 11; 15; 1; 3; 9)$\\
\hline
$101$ & $(4; 11; 15; 1; 2; 8)$ &$(3; 9; 10; 5; 7; 8)$ &$(3; 9; 10; 4; 6; 8)$ &$(3; 9; 10; 2; 7; 15)$ &$(3; 9; 10; 2; 6; 14)$\\
\hline
$102$ & $(2; 8; 10; 4; 7; 9)$ &$(2; 8; 10; 3; 7; 14)$ &$(7; 9; 14; 4; 5; 8)$ &$(7; 9; 14; 3; 5; 15)$ &$(7; 9; 14; 1; 5; 13)$\\
\hline
 \end{tabular}
\end{center}}
\noindent In the above table, each $(a; b; c; d; e; l)$ means a
polynomial $d(x_ax_bx_c)d(x_dx_ex_l)$ in $\mathcal{V}_4^*$, where
$x_1, x_2, ..., x_{15}$ denote the 15 nontrivial elements in
$\Hom(\Z_2,(\Z_2)^4)$ with
$$
\begin{cases}
x_1=\rho^*_1 & x_2=\rho^*_2 \\
 x_3=\rho^*_1+\rho^*_2 & x_4=\rho^*_3\\
x_5=\rho^*_1+\rho^*_3 &x_6=\rho^*_2+\rho^*_3 \\
x_7=\rho^*_1+\rho^*_2+\rho^*_3& x_8=\rho^*_4\\
x_9=\rho^*_1+\rho^*_4 &x_{10}=\rho^*_2+\rho^*_4\\
 x_{11}=\rho^*_1+\rho^*_2+\rho^*_4 &x_{12}=\rho^*_3+\rho^*_4\\
x_{13}=\rho^*_1+\rho^*_3+\rho^*_4 & x_{14}=\rho^*_2+\rho^*_3+\rho^*_4 \\
 x_{15}=\rho^*_1+\rho^*_2+\rho^*_3+\rho^*_4 & \text{$\{\rho^*_1, \rho^*_2, \rho^*_3, \rho^*_4\}$ is the standard basis of $\Hom(\Z_2,(\Z_2)^4)$.}
\end{cases}$$

\vskip .3cm

\noindent{\bf Acknowledgment.} The authors would like to express their gratitude to the referees, who did an extremely careful reading of the original version and many
valuable suggestions, comments and corrections.

\end{document}

%% file: n2.pstex_t
\begin{picture}(0,0)%
\includegraphics{n2.pstex}%
\end{picture}%
\setlength{\unitlength}{1579sp}%
\begingroup\makeatletter\ifx\SetFigFont\undefined%
\gdef\SetFigFont#1#2#3#4#5{%
  \reset@font\fontsize{#1}{#2pt}%
  \fontfamily{#3}\fontseries{#4}\fontshape{#5}%
  \selectfont}%
\fi\endgroup%
\begin{picture}(9450,3447)(2026,-9733)
\put(4501,-7486){\makebox(0,0)[lb]{\smash{{\SetFigFont{7}{8.4}{\rmdefault}{\mddefault}{\updefault}$\rho_2$}}}}
\put(9376,-9136){\makebox(0,0)[lb]{\smash{{\SetFigFont{7}{8.4}{\rmdefault}{\mddefault}{\updefault}$\rho_2+\rho_3$}}}}
\put(8851,-9661){\makebox(0,0)[lb]{\smash{{\SetFigFont{7}{8.4}{\rmdefault}{\mddefault}{\updefault}The case $n=3$}}}}
\put(2551,-9661){\makebox(0,0)[lb]{\smash{{\SetFigFont{7}{8.4}{\rmdefault}{\mddefault}{\updefault}The case $n=2$}}}}
\put(9526,-7336){\makebox(0,0)[lb]{\smash{{\SetFigFont{7}{8.4}{\rmdefault}{\mddefault}{\updefault}$\rho_1$}}}}
\put(9001,-8086){\makebox(0,0)[lb]{\smash{{\SetFigFont{7}{8.4}{\rmdefault}{\mddefault}{\updefault}$\rho_2$}}}}
\put(3001,-9136){\makebox(0,0)[lb]{\smash{{\SetFigFont{7}{8.4}{\rmdefault}{\mddefault}{\updefault}$\rho_1+\rho_2$}}}}
\put(10351,-8011){\makebox(0,0)[lb]{\smash{{\SetFigFont{7}{8.4}{\rmdefault}{\mddefault}{\updefault}$\rho_3$}}}}
\put(2251,-7486){\makebox(0,0)[lb]{\smash{{\SetFigFont{7}{8.4}{\rmdefault}{\mddefault}{\updefault}$\rho_1$}}}}
\put(7801,-7336){\makebox(0,0)[lb]{\smash{{\SetFigFont{7}{8.4}{\rmdefault}{\mddefault}{\updefault}$\rho_1+\rho_2$}}}}
\put(10726,-7336){\makebox(0,0)[lb]{\smash{{\SetFigFont{7}{8.4}{\rmdefault}{\mddefault}{\updefault}$\rho_1+\rho_3$}}}}
\end{picture}%

%% file: n3.pstex_t
\begin{picture}(0,0)%
\includegraphics{n3.pstex}%
\end{picture}%
\setlength{\unitlength}{1579sp}%
\begingroup\makeatletter\ifx\SetFigFont\undefined%
\gdef\SetFigFont#1#2#3#4#5{%
  \reset@font\fontsize{#1}{#2pt}%
  \fontfamily{#3}\fontseries{#4}\fontshape{#5}%
  \selectfont}%
\fi\endgroup%
\begin{picture}(3975,3381)(8401,-6367)
\put(10201,-5536){\makebox(0,0)[lb]{\smash{{\SetFigFont{8}{9.6}{\rmdefault}{\mddefault}{\updefault}$\rho_1+\rho_2$}}}}
\put(12151,-3961){\makebox(0,0)[lb]{\smash{{\SetFigFont{8}{9.6}{\rmdefault}{\mddefault}{\updefault}$\rho_2+\rho_3$}}}}
\put(11251,-4486){\makebox(0,0)[lb]{\smash{{\SetFigFont{8}{9.6}{\rmdefault}{\mddefault}{\updefault}$\rho_2$}}}}
\put(8401,-3961){\makebox(0,0)[lb]{\smash{{\SetFigFont{8}{9.6}{\rmdefault}{\mddefault}{\updefault}$\rho_1+\rho_3$}}}}
\put(9976,-4486){\makebox(0,0)[lb]{\smash{{\SetFigFont{8}{9.6}{\rmdefault}{\mddefault}{\updefault}$\rho_1$}}}}
\put(9976,-6286){\makebox(0,0)[lb]{\smash{{\SetFigFont{8}{9.6}{\rmdefault}{\mddefault}{\updefault}$\rho_1+\rho_2+\rho_3$}}}}
\end{picture}%

%% file: l.pstex_t
\begin{picture}(0,0)%
\includegraphics{l.pstex}%
\end{picture}%
\setlength{\unitlength}{1381sp}%
\begingroup\makeatletter\ifx\SetFigFont\undefined%
\gdef\SetFigFont#1#2#3#4#5{%
  \reset@font\fontsize{#1}{#2pt}%
  \fontfamily{#3}\fontseries{#4}\fontshape{#5}%
  \selectfont}%
\fi\endgroup%
\begin{picture}(10608,6147)(2068,-16567)
\put(9751,-11836){\makebox(0,0)[lb]{\smash{{\SetFigFont{6}{7.2}{\rmdefault}{\mddefault}{\updefault}$\rho_3$}}}}
\put(3151,-16486){\makebox(0,0)[lb]{\smash{{\SetFigFont{6}{7.2}{\rmdefault}{\mddefault}{\updefault}$(P^3(3), \lambda)$}}}}
\put(10351,-16486){\makebox(0,0)[lb]{\smash{{\SetFigFont{6}{7.2}{\rmdefault}{\mddefault}{\updefault}$(\Gamma_M, \alpha)$}}}}
\put(3601,-11386){\makebox(0,0)[lb]{\smash{{\SetFigFont{6}{7.2}{\rmdefault}{\mddefault}{\updefault}$\rho^*_1$}}}}
\put(5626,-11461){\makebox(0,0)[lb]{\smash{{\SetFigFont{6}{7.2}{\rmdefault}{\mddefault}{\updefault}$\rho_2^*+\rho^*_3$}}}}
\put(2701,-13186){\makebox(0,0)[lb]{\smash{{\SetFigFont{6}{7.2}{\rmdefault}{\mddefault}{\updefault}$\rho^*_2$}}}}
\put(4276,-13186){\makebox(0,0)[lb]{\smash{{\SetFigFont{6}{7.2}{\rmdefault}{\mddefault}{\updefault}$\rho^*_3$}}}}
\put(5401,-15661){\makebox(0,0)[lb]{\smash{{\SetFigFont{6}{7.2}{\rmdefault}{\mddefault}{\updefault}$\rho^*_1+\rho^*_2$}}}}
\put(10351,-10636){\makebox(0,0)[lb]{\smash{{\SetFigFont{6}{7.2}{\rmdefault}{\mddefault}{\updefault}$\rho_2+\rho_3$}}}}
\put(11476,-11836){\makebox(0,0)[lb]{\smash{{\SetFigFont{6}{7.2}{\rmdefault}{\mddefault}{\updefault}$\rho_2$}}}}
\put(8776,-13036){\makebox(0,0)[lb]{\smash{{\SetFigFont{6}{7.2}{\rmdefault}{\mddefault}{\updefault}$\rho_1$}}}}
\put(10351,-13561){\makebox(0,0)[lb]{\smash{{\SetFigFont{6}{7.2}{\rmdefault}{\mddefault}{\updefault}$\rho_1$}}}}
\put(9451,-15511){\makebox(0,0)[lb]{\smash{{\SetFigFont{6}{7.2}{\rmdefault}{\mddefault}{\updefault}$\rho_3$}}}}
\put(12376,-13036){\makebox(0,0)[lb]{\smash{{\SetFigFont{6}{7.2}{\rmdefault}{\mddefault}{\updefault}$\rho_1$}}}}
\put(11701,-15511){\makebox(0,0)[lb]{\smash{{\SetFigFont{6}{7.2}{\rmdefault}{\mddefault}{\updefault}$\rho_1+\rho_2$}}}}
\put(12676,-14611){\makebox(0,0)[lb]{\smash{{\SetFigFont{6}{7.2}{\rmdefault}{\mddefault}{\updefault}$\rho_1+\rho_2+\rho_3$}}}}
\end{picture}%